\documentclass[11pt, letterpaper]{article}
\usepackage[margin=1in]{geometry}

\usepackage[round]{natbib}
\bibpunct[, ]{(}{)}{,}{a}{}{,}
\usepackage[utf8]{inputenc}
\usepackage{ifthen}
\usepackage{float}
\usepackage{graphicx}
\usepackage{nomencl}
\usepackage{authblk}
\usepackage[hidelinks]{hyperref}

\usepackage{amsmath, amssymb, amsthm}
\newtheorem{theorem}{Theorem}
\newtheorem{lemma}{Lemma}
\newtheorem{proposition}{Proposition}
\newtheorem{corollary}{Corollary}
\newtheorem{definition}{Definition}

\usepackage{cleveref}

\theoremstyle{remark}
\newtheorem{remark}{Remark} 

\usepackage{pgfplots}
\pgfplotsset{compat=1.18} 
\usepackage{tikz}
\usetikzlibrary{positioning}
\usetikzlibrary{shapes.geometric}
\usetikzlibrary{shapes.multipart}
\usetikzlibrary{decorations.text}
\usetikzlibrary{decorations.pathmorphing}

\newcommand{\R}{\mathbb{R}}

\makenomenclature

\renewcommand{\nomgroup}[1]{%
  \ifthenelse{\equal{#1}{C}}{\item[\textbf{Constants}]}{%
  \ifthenelse{\equal{#1}{D}}{\item[\textbf{DC-OPF Dual Variables}]}{%
  \ifthenelse{\equal{#1}{V}}{\item[\textbf{DC-OPF Variables}]}{%
  \ifthenelse{\equal{#1}{S}}{\item[\textbf{Sets}]}{}}}}%
}

\title{On Locational Marginal Emissions in Electricity Markets: A Two-Layered Dispatch Mechanism and Its Fundamental Theorems}

\author[1]{Andy Sun\thanks{sunx@mit.edu}}
\affil[1]{Sloan School of Management, Massachusetts Institute of Technology, Cambridge, MA}

\author[2]{Luc Cote\thanks{luccote@mit.edu}}
\affil[2]{Operations Research Center, Massachusetts Institute of Technology, Cambridge, MA}

\date{} 
\allowdisplaybreaks
\begin{document}

\maketitle

\begin{abstract}
We propose a market design for real-time electricity markets that utilizes a two-layered dispatch mechanism to systematically incorporate carbon accounting into grid operations. In this mechanism, ``dispatch'', the centralized allocation of generation resources to meet system load, is executed via a hierarchical structure where the first layer minimizes financial costs to maintain economic efficiency, while the second layer minimizes system emissions strictly within the set of cost-optimal solutions. We define locational marginal emissions (LMEs) as the marginal rate of system emissions derived from the dual variables of the two-layered formulation. Unlike standard marginal prices which correspond to right-hand-side constraint relaxations, LMEs must account for the requirement of economic optimality which introduces demand parameters into the problem's constraint structure. Under the framework, we establish that LMEs satisfy properties analogous to the first and second fundamental theorems of welfare economics. We prove that (1) decentralized ``carbon profit'' maximization by individual grid entities guarantees a system-wide emission profile consistent with the economic dispatch, and (2) any optimal low-carbon economic dispatch is supported by a corresponding set of LME signals acting as a decentralized equilibrium. Furthermore, we establish a general carbon accounting theorem, called the Carbon Footprint Theorem, showing that these market-consistent LMEs ensure the sum of carbon accounts across all grid components (loads, generators, transmission, and storage) equals the total physical carbon emissions. This completes the theoretical foundation of the LME. Finally, we investigate and validate the empirical properties of LMEs and LME-based carbon accounting through case studies on a realistic Texas grid model.
\end{abstract}



%


\section{Introduction}
Accounting for over 30\% of worldwide greenhouse gas emissions, the electricity sector represents a critical frontier for global climate targets \citep{ipcc_emissions_2023}.
Industrial and commercial electricity usage is currently responsible for approximately half of all grid demand \citep{iea_electricity_2021} and is likely to only grow as artificial intelligence and energy-intensive data centers continue to develop \citep{cao_toward_2022}. Leveraging such key stakeholders has the potential to significantly accelerate decarbonization pathways, giving the proper alignment of corporate incentives and grid operations a major role in creating a clean energy future. Such incentive systems typically work through principles of carbon accounting, particularly in the form of scope two emissions accounting which seeks to attribute indirect emissions to the consumption of electricity, heating, and cooling. 

\subsection{Issues in GHG Accounting} \label{sec:issues}
The Greenhouse Gas Protocol currently outlines two distinct methods for accounting of scope two emissions from electricity consumption: location-based and market-based \citep{ghgscope2}. The location-based method relies on calculating a grid average emissions factor for the consumer's region over a certain (usually lengthy) time period and multiplies this by the consumer's energy over the period. The market-based approach allows the consumer to make choices regarding their electricity supplier through the use of energy market mechanisms such as particular retail suppliers or the purchase of energy attribute certificates. Such market-based approaches can be used by consumers to decrease their electricity-derived scope two emissions even if they are located within a high-emission region.
While these methods represent progress, the current corporate accounting practices outlined in the GHG protocol have been the subject of criticism from key stakeholders including both corporations and researchers. The main challenges faced by such an accounting scheme may be summarized by the following four issues: \textit{(1)} deliverability, \textit{(2)} double counting, \textit{(3)} additionality, and \textit{(4)} impact magnitude \citep{gillenwater_redefining_2008, brander_creative_2018, bjorn_renewable_2022, Brander}. 

Implicit in each of these issues is a disconnect in accounting regimes between traditional location-based carbon account allocation and the impact of market-based interventions. In seeking to address such a disconnect we consider how a unified accounting approach based on the causal nature of consumption, production, and interventions may improve upon principles of deliverability, double counting, and impact magnitude. Such causality based factors are typically referred to as locational marginal emissions, or hereafter, LMEs.

\subsection{Literature Review}\label{sec:litrev}
LMEs have seen growing attention from literature in recent years, with the evaluation and quantification of decarbonization becoming increasingly important. Studies typically use LMEs for the evaluation of system interventions and consumer actions \citep{jiang_can_2025,wang_locational_2014,nilges_validity_2025, he_using_2021}, but have also been considered for applications in carbon accounting and grid planning frameworks \citep{rudkevich_locational_2012,xu_system-level_2024}. Such interest in applications has additionally led to the publishing of real-world LME data, most notably by PJM \citep{pjm_five_2025}.

 Despite significant usage in literature and applications, the methodology behind the calculation of LMEs remains varied, generally falling into two categories: empirical \citep{xu_system-level_2024, he_is_2024, hawkes_estimating_2010, siler-evans_marginal_2012, marnay_estimating_2002} and analytical \citep{bettle_interactions_2006, valenzuela_dynamic_2024,ruiz2010analysis}. Empirical approaches estimate emissions responses, often via regression or counterfactual simulations, while useful for prediction, these methods are unable to provide a technical and accurate definition for LMEs. In contrast, analytical methods derive LMEs directly from power system models. The standard OPF-based framework introduced by \cite{ruiz2010analysis} allows for precise nodal definitions but relies on assumptions of OPF solution uniqueness and non-degeneracy. In experimental or planning settings which may require assumptions about grid data (such as generator cost) and model such assumptions on shared characteristics (such as generator type), it can be common for non-unique solutions to exist. Indeed, \cite{gorka_electricityemissionsjl_2025} notes that this may be an issue in practice and provides a method for artificially perturbing power flow dispatch solutions with small amounts of cost noise to generate such solutions. Furthermore, such approaches rely on post-dispatch knowledge of the discrete set of marginal generators and tight lines. Thus, if one wishes to integrate LME calculation into optimization problems such as capacity expansion, this would necessitate the use of integer variables and big-M formulations to track such a set in addition to the already bilinear definition of LMEs.

In addition to the varied literature around calculation methodologies, the market mechanisms and decentralized properties of LME-based schemes have remained understudied. In the realm of power system economics, the widespread adoption of Locational Marginal Pricing (LMP) is driven by its alignment with the fundamental theorems of welfare economics \citep{Microeconomic_Theory}. Specifically, LMPs satisfy two critical properties necessary for efficient market operation \citep{hogan1998competitive, jiang2023duality}:
\begin{theorem}[Social Welfare of LMPs]\label{thm:lmp:sw}
    \item \textbf{Sufficiency:} Analogous to the First Fundamental Theorem of Welfare Economics, LMP pricing ensures that when individual market participants (generators) maximize their own profits against these prices, the resulting dispatch supports the centralized social welfare maximum (minimum system cost).  \label{thm:lmp:fft}
    \item \textbf{Existence:} Analogous to the Second Fundamental Theorem, for any efficient centralized dispatch, there exists a set of LMP prices that support this dispatch as a decentralized equilibrium.\label{thm:lmp:sft}
\end{theorem}
These properties have important implications, incentivizing resources to act consistently with the system operator's central dispatch, significantly decreasing the need for forcing mechanisms or non-cooperation fines. When combined with cost additivity (where financial accounts sum to total system cost), they form the bedrock of modern electricity markets. In the context of emissions, \cite{ruiz2010analysis} demonstrated that LMEs can exhibit a similar additivity property, showing that carbon accounts sum to total physical system emissions, once again requiring assumptions of DC-OPF uniqueness and nondegeneracy. However, there has been little literature addressing whether carbon accounting schemes can support a decentralized equilibrium compatible with current pricing schemes.

\subsection{Contributions}
This paper makes the following three main contributions:

\begin{enumerate}
    \item \textit{Two-Layered Dispatch Mechanism:} In \Cref{sec:model} we propose a two-layered dispatch mechanism as a novel market design for real-time electricity markets. By hierarchically structuring the market clearing process to minimize emissions within the set of cost-optimal solutions, this framework moves beyond ex-post accounting to provide a rigorous, optimization-based definition of Locational Marginal Emissions in Definition \ref{def:extlme}. This design ensures LMEs remain well-defined at any grid state and easily adapts to linear dispatch variants, with \Cref{sec:storage} describing one possible adaptation to include storage resources.
    \item \textit{Two-Layered Market Equilibria and Properties:} \Cref{thm:fft} and \Cref{thm:sft} introduce the first and second fundamental theorems of LME-based carbon mechanisms respectively, analogous to the first and second fundamental theorems of welfare economics, and show that grid entities are properly incentivized to follow the lowest emission economically optimal central dispatch. In \Cref{thm:carbonfootprint} we provide a novel proof for the additivity property of LME-based accounting, showing that the property generalizes to the introduced calculation methodology. Finally, through \Cref{eq:carbonfootprintstorage},\Cref{thm:stor-fft}, and \Cref{thm:stor-sft} we show that such properties also generalize to common economic variants such as the inclusion of storage.
    \item \textit{Empirical Properties:} We apply the novel LME calculation methodology to a 385-node realistic representation of the ERCOT grid with hourly dispatches simulated over 1 year of generation and load data. Through such experiments, we analyze the temporal and spatial patterns of LMEs within the ERCOT grid and how such patterns are reflected across carbon accounts. We find the Texas grid to be naturally split into a high renewable generation western section and high consumption eastern section. LMEs within the western portion exhibited consistently low average LMEs while the eastern section demonstrated daily and seasonal temporal trends consistent with those of solar production strength.
\end{enumerate}

\section{Two-layered Dispatch Mechanism for the Calculation and Usage of Locational Marginal Emissions}
In this section we develop theory and intuition around the usage of LMEs. We describe the underlying techniques governing grid dispatch and relate such techniques to system emissions through the usage of a two-layered dispatch model in \Cref{sec:model}. Such descriptions of system emissions are then used in \Cref{sec:lmes} to introduce a principled definition of LMEs and subsequent novel calculation methodology. Finally, in \Cref{sec:accounting} we develop an LME-based carbon accounting system for grid entities analogous to their economic LMP accounts.
\subsection{Economic Dispatch and Marginal Units}\label{sec:ed}
To understand how the emissions of an energy grid can vary with changes in demand, we begin by describing how grid resources are dispatched. In particular, we consider an economic dispatch-driven power grid scheduling process. 

In the simplest case, consider a ``copper plate'' grid that operates without transmission constraints. In this scenario, generation resources are dispatched in least production cost order until all demand is satisfied. In such a scenario, the last unit dispatched, i.e., the highest cost dispatched unit, is considered the marginal unit as an increase or decrease in demand will result in an increase or decrease in this unit's dispatch level. Further complexities surrounding marginal generators can arise. If the last unit is dispatched fully (i.e., to its max power limit), then a decrease in demand would lead to a decrease in this unit's dispatch, but an increase in demand would result in an increase in the dispatch of another generation unit lower on the cost dispatch ranking. In such cases, notions of marginal change such as marginal price or marginal emissions may cease to be smooth, necessitating the use of directional derivatives.

However, most electrical grids are not well characterized by copper plate assumptions and in practice economic dispatch is typically performed through optimal power flow problems. In particular, the DC-OPF problem (and many of its variations) is used to set generation levels to fulfill power demand in a least cost fashion while respecting transmission constraints and power flow laws. 

We define the DC-OPF problem using the following sets: let $N$ denote the set of grid nodes, $G$ the set of generators, $G_i$ the set of generators at node $i$, and $L$ the set of transmission lines. The decision variables are the generation level $P^G_g$ for generator $g$, the power flow $f_\ell$ on line $\ell$, and the voltage angle $\theta_i$ at node $i$. Finally, we have the following constraints on the variable set:
\begin{itemize}
    \item each line $\ell=(i,j)$ has a flow set by the line's admittance, characterized as $B_{ij}$, multiplied by the difference in voltage angles across the line, defined by $\theta_i-\theta_j$ \eqref{eq:dcopf:angle},
    \item each grid node must have balanced input and output power flows \eqref{eq:dcopf:fb},
    \item each line, $\ell$, must not exceed its maximum flow rating of $f^{max}_\ell$ \eqref{eq:dcopf:fmax} \eqref{eq:dcopf:fmin},
    \item each generator, $g$, must not violate its maximum generation rating of $P^{max}_g$ and its minimum generation rating of $P^{min}_g$ \eqref{eq:dcopf:gmax} \eqref{eq:dcopf:gmin}.
\end{itemize}
This leads to the following optimization problem
\begin{align}
   DCOPF(P^D):= \min_{P^G, \theta, f} \quad  & c^T P^G \tag{DCOPF Model} \label{p:dcOPF}\\
    \text{s.t.}\quad &
    B_{ij} (\theta_i - \theta_j) - f_\ell =  0, &\forall \ell=(i,j) \in L,&~[z_\ell],\label{eq:dcopf:angle}\\
   &  \sum_{g \in G_i} P^G_g - \sum_{\ell = (i,j) \in L}f_\ell + \sum_{\ell = (j,i) \in L}f_\ell - P_i^D= 0, & \forall i \in N,&~[\pi_i], \label{eq:dcopf:fb}\\
    & f_\ell \leq f^{max}_\ell, & \forall \ell \in L,&~[\rho^+_\ell],  \label{eq:dcopf:fmax}\\
    & f_\ell \geq -f^{max}_\ell, & \forall \ell \in L,&~[\rho^-_\ell],  \label{eq:dcopf:fmin}\\
    & P^G_g \leq P^{max}_g, & \forall g \in G,&~[\gamma^+_g],\label{eq:dcopf:gmax}\\
    & P^G_g \geq P^{min}_g, & \forall g \in G,&~[\gamma^-_g]\label{eq:dcopf:gmin}.
\end{align} 
where we use the variables in right side brackets to denote each constraint's corresponding Lagrangian dual variable to be used in future sections.

Similar to the copper plate case, we can define marginal units as those generation units which respond to increases or decreases in demand to maintain a feasible optimal solution. In \cite{ruiz2010analysis}, it is shown that under assumptions of solution uniqueness and nondegeneracy the set of marginal units is precisely those generators whose generation level at optimality satisfies
\begin{equation}
    P^{min}_g < P^G_g<P^{max}_g.
\end{equation}

Additionally, to characterize the response of the marginal generators, \cite{ruiz2010analysis} use an electrical engineering analysis term known as the Power Transfer Distribution Factor (PTDF) matrix, denoted as $\Psi$. $\Psi$ represents the Jacobian matrix of the transmission network flows. Specifically, an element $\Psi_{l,n}$ represents the linear sensitivity (or partial derivative) of the flow on transmission line $l$ with respect to a unit increase in power injection at node $n$. Using such a term, the corresponding weighted response of the vector of marginal generators with respect to a change in demand at node $n$ can be described as 
\begin{equation}
    \begin{bmatrix}
    \vec{1} \mid
    \tilde{\Psi}^T
\end{bmatrix}^{-1} \begin{bmatrix}
    \vec{1} \\
    \tilde{\psi}_n
\end{bmatrix},
\end{equation}
where $\tilde{\Psi}$ describes a subset of the power transfer distribution factor matrix, $\Psi$, which contains only columns corresponding to nodes with marginal generators (generators which are not maximally or minimally dispatched) and transmission lines with tight flow constraints, and where $\tilde{\psi}_n$ similarly describes a subset of the $n$th column of $\Psi$ which contains only elements from rows corresponding to transmission lines with tight flow constraints. However, just as in the copper plate case, one should note that without assumptions of uniqueness and non degeneracy, there may be distinct direction-dependent responses to changes in nodal demand.

\subsection{Combined Model}\label{sec:model}
In seeking an improved analytical definition, we introduce a novel method for LME calculation which we term the \emph{combined model}. We formulate an emission-aware economic dispatch problem that minimizes system carbon emissions over the set of least-cost dispatch solutions:
\begin{align}
   E(P^D):= \min_{P^G, \theta, f}\quad  & \sigma^T P^G \tag{Combined Model} \label{p:OPF}\\
    \text{s.t.}\quad &   (P^G, f, \theta) \in X^*(P^D),
\end{align}
where $X^*(P^D)$ is the set of optimal dispatch solutions to the standard economic dispatch problem with DC power flow constraints (\ref{p:dcOPF}). By dualizing the standard DC OPF problem and constraining strong duality via primal and dual objectives, we can describe the economically optimal dispatch set $X^*(P^D)$ with a finite number of linear constraints as
{
\begin{align}
    &B_{ij} (\theta_i - \theta_j) - f_\ell =  0, &\forall \ell=(i,j) \in L, &~[p_{z,\ell}],\label{eq:cm:kvl}\\
    &\sum_{g \in G_i} P^G_g - \sum_{\ell = (i,j) \in L}f_\ell + \sum_{\ell = (j,i) \in L}f_\ell - P_i^D= 0, & \forall i \in N, &~[p_{\pi,i}],\label{eq:cm:fb}\\
    &f_\ell \leq f^{max}_\ell, & \forall \ell \in L,&~[p^+_{\rho,\ell}], \label{eq:cm:fp}\\
    &f_\ell \geq -f^{max}_\ell & \forall \ell \in L,&~[p^-_{\rho,\ell}], \label{eq:cm:fm}\\
    &P^G_g \leq P^{max}_g, & \forall g \in G, &~[p^+_{\gamma,g}],\\
    &P^G_g \geq P^{min}_g, & \forall g \in G, &~[p^-_{\gamma,g}],\\
    &\gamma^+_g + \gamma^-_g + \sum_{i \mid g \in G_i} \pi_i  = c_g, &\forall g \in G, \\
    &-z_\ell + \rho^+_\ell + \rho^-_\ell - \pi_i + \pi_j = 0, &\forall \ell=(i,j)\in L, \\
    &\sum_{\ell = (i,j) \in L} B_\ell z_\ell - \sum_{\ell = (j,i) \in L} B_\ell z_\ell = 0, &\forall i \in N, \\
    &\rho^+, \gamma^+ \leq 0, \\
    &\rho^-, \gamma^- \geq 0, \\
    &c^TP^G - \pi^T P^D - (\rho^+ - \rho^-)^Tf^{max} - (\gamma^+)^T P^{max} - (\gamma^-)^T P^{min} \le 0, &&~[p_o]. \label{eq:cm:opt}
\end{align}
}
Casting the problem in such a way allows us to capture how system carbon output will change by analyzing how the objective function will change with respect to parameter perturbations. However, note that by constraining the objective using strong duality, we have changed the nature of the demand parameter $P^D$ from being a right hand side constraint of the linear program to being a decision variable coefficient in constraint \eqref{eq:cm:opt}. This complicates the analysis of how perturbations in load will impact the overall optimization problem. Fortunately, such analysis problems have seen previous study and we can use a characterization of sensitivity analysis stemming from \cite{freund2009postoptimal} to describe changes in optimization problem value with respect to changes in  decision coefficients\footnote{\Cref{thm:freund} is introduced in \cite{freund2009postoptimal}, however, for completeness we include a generalized proof pertaining to a directional behavior as well in \Cref{sec:sensan}.}.

\begin{lemma}\label{thm:freund}
    Consider a linear optimization problem of the form:
    \begin{align}
        \min \quad &c^T x\\
        \text{s.t. }
        &Ax = b, &[p],\\
        &x \geq 0,
    \end{align}
    with $A\in \R^{m\times n}$. Let $P(\alpha)$ denote the perturbed problem where $A$ is replaced by $(A+\alpha G)$ for an arbitrary matrix $G\in \R^{m\times n}$. Let $z(\alpha)$ denote the optimal value of $P(\alpha)$. Then, for the optimal primal and dual solutions $x^*$ and $p^*$ of the problem $P(\alpha)$, the following holds:
    \begin{equation}
        z'(\alpha)=-(p^*)^TGx^*.
    \end{equation}
\end{lemma}

\subsubsection{Properties of the Emissions Function}
To understand how grid-level actions influence emissions, we analyze how the emission function $E(P^D)$ responds to variations in demand at node $i$. We define the nodal emissions function as:
\begin{equation}
    f_i(\alpha) := E(P^D+\alpha e_i), 
\end{equation}
where $e_i$ is a vector of length $|N|$ with a one in the $ith$ position and zeros everywhere else. For clarity, we restrict the domain of $E$ to the set of $\hat{P}^D$ such that $X^*(\hat{P}^D)$ is nonempty, i.e., there exists a solution to the \eqref{p:dcOPF}, hereafter referred to as $D$. Similarly, we restrict the domain of $f_i$ to be values of $\alpha$ such that $P^D+\alpha e_i \in D$, hereafter referred to as $D_i$.

\begin{proposition}\label{prop:contpiece}
    The function $f_i:D_i\rightarrow\mathbb{R}$ is a continuous, piecewise linear function on $D_i$.
\end{proposition}

\begin{proof}

    We observe that the Combined Model is a hierarchical linear program where the parameter $\alpha$ perturbs the demand $P^D$. While $P^D$ appears as a right-hand side constant in the flow balance constraints \eqref{eq:dcopf:fb}, it acts as a coefficient in the strong duality constraint \eqref{eq:cm:opt}, altering the constraint matrix of the \ref{p:OPF} with respect to the dual variables. To establish the piecewise linear property despite this matrix perturbation, we utilize the two-layered structure of the problem. The domain $D_i$ can be partitioned into a finite number of intervals based on the optimal bases of the \ref{p:dcOPF}. Within any such interval, the set of optimal bases for the \ref{p:dcOPF} remains invariant. Consequently, the feasible set of the \ref{p:OPF} (which corresponds to the optimal face of the \ref{p:dcOPF}) evolves linearly with $\alpha$, and so too does the optimal value to the \ref{p:OPF}. 
    
    To establish continuity between intervals, consider a boundary point $\bar{\alpha}$ separating two intervals. Since the optimality conditions for a basis to the \ref{p:dcOPF} consist of closed linear inequalities, a basis that is optimal on the adjacent open interval remains optimal at the boundary $\bar{\alpha}$. This ensures that the linear value function defined by this basis extends to $\bar{\alpha}$ as a valid feasible solution for the \ref{p:OPF}. Furthermore, from the linearity of the problem no strictly better solution can emerge discontinuously at the boundary. If such a solution existed at $\bar{\alpha}$, it would necessarily be superior within the adjacent interval's neighborhood, contradicting the optimality of the original basis and thus, the function limits from both sides must converge to the function value at $\bar{\alpha}$, leading to the continuity of $f_i$ on $D_{i}$.\footnote{For the interested reader, we provide further discussion and intuition on the properties of the \ref{p:OPF} discussed in the above proof of Proposition \ref{prop:contpiece} in \Cref{sec:combinedProp}}
\end{proof}

\begin{corollary} \label{corr:unique}
    Under assumptions of primal and dual solution uniqueness of the \ref{p:OPF}, $f_i$ is differentiable.
\end{corollary}

    However, while $f_i(\alpha)$ does enjoy continuity and piecewise linear properties, we note that we are unable to make guarantees regarding its convexity/concavity. This can be shown by counterexample through considering a simple single node, copper plate instance where we have three generation resources: renewable (cheapest cost, low carbon), coal (medium cost, high carbon), and natural gas (highest cost, medium carbon).
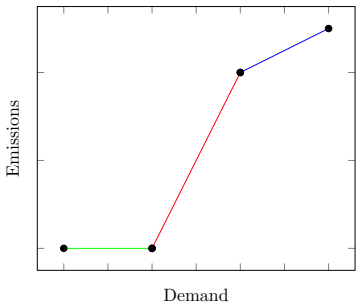
\begin{figure}[H]
    \centering
    \resizebox{0.3\textwidth}{!}{
        \begin{tikzpicture}
\begin{axis}[
yticklabel=\empty, ylabel=Emissions,
xticklabel=\empty, xlabel=Demand
]
\addplot+ [
sharp plot,
color=green,
mark=*,
mark size=2pt,
mark options={fill=black,color=black}
] coordinates {
(0,0) (1,0)
};
\addplot+ [
sharp plot,
color=red,
mark=*,
mark size=2pt,
mark options={fill=black,color=black}
] coordinates {
(1,0) (2,2)
};
\addplot+ [
sharp plot,
color=blue,
mark=*,
mark size=2pt,
mark options={fill=black,color=black}
] coordinates {
(2,2) (3,2.5)
};
\end{axis}
\end{tikzpicture}
        }
    \caption{Example emissions function displaying both local concavity and convexity. Green represents the section where the renewable plant is on the margin, red where coal is on the margin, and blue where natural gas is on the margin.}
    \label{fig:emissionsfunc}
\end{figure}
\Cref{fig:emissionsfunc} also highlights a second key property of the emissions function - its differentiability (or lackthereof). In this simple example, at the junctions between between marginal generation resources we see that the rate of change differs depending on the direction of perturbation. This will become an important consideration as we apply the lens of LMEs in the following subsection.

\subsection{LME Characterization}\label{sec:lmes}
We follow previous notions of locational marginal emissions and conceptually define the nodal LME to be the change in total system emissions with respect to an increase in nodal demand. 

\begin{definition}[Classical LMEs]\label{def:lme}
\begin{equation} 
    LME_i:=\frac{d E(P^D)}{d P^D_i}=f_i'(0),
\end{equation}
\end{definition}

Applying Corollary \ref{corr:unique} allows us to observe that such an LME must always exist under assumptions of primal and dual solution uniqueness for the \ref{p:OPF}. 
We can thus use the description of linear program perturbation analysis of Lemma \ref{thm:freund} to develop \Cref{thm:LME}.

\begin{theorem} \label{thm:LME}
    Under assumptions of primal and dual solution uniqueness to the \ref{p:OPF}
    \begin{align}
        \frac{d E(P^D)}{d P^D_i}=p_{\pi,i}+p_o\pi_i,
    \end{align}
    where $\pi_i$ refers to the optimal primal variable within the \ref{p:OPF}, $p_{\pi,i}$ refers to the optimal Lagrangian dual variable corresponding to \eqref{eq:cm:fb} for node $i$, and $p_{o}$ refers to the optimal Lagrangian dual variable corresponding to \eqref{eq:cm:opt}.
\end{theorem}

Intuitively, one would expect that when assumptions of primal and dual solution uniqueness are met for the \ref{p:OPF}, the previously derived LME equation from \cite{ruiz2010analysis} and \Cref{thm:LME} should agree as the set $X^*(P^D)$ will be described by a singleton with a corresponding singleton set of optimal dual solutions and so adding the additional carbon-aware objective cannot change the dispatch description. Indeed, in \Cref{s:CombinedEQ} of \Cref{sec:combinedProp} we show
\begin{corollary}
    Under assumptions of primal and dual solution uniqueness to the \ref{p:dcOPF}, 
    \begin{equation}
    p_{\pi,i}+p_o\pi_i = \sigma^T \begin{bmatrix}
    \vec{1} \mid
    \tilde{\Psi}^T
\end{bmatrix}^{-1} \begin{bmatrix}
    \vec{1} \\
    \tilde{\psi}_i
\end{bmatrix}.
\end{equation}
Thus, \Cref{thm:LME} is equivalent to the calculation methodology proposed in \cite{ruiz2010analysis}.
\end{corollary}

However, as described by \Cref{prop:contpiece}, the function of system emissions is continuous but need not always be smooth. In such cases where we lack primal and dual solution uniqueness and increases/decreases in nodal demand correspond to different changes in emissions, the derivative of the emissions function will be undefined. 

Such issues of non-differentiability have also been seen in locational marginal pricing, where in instances of the \ref{p:dcOPF} which lack dual solution uniqueness, the economic cost function may be non differentiable. In such cases, there are many possible pricing schemes which lead to a market clearing equilibrium. Indeed, it can be shown that any marginal pricing scheme defined using an optimal dual solution to the \ref{p:dcOPF} will support a market clearing equilibrium, leading to the convention of considering the locational marginal price to be defined as any optimal dual solution corresponding to \eqref{eq:dcopf:fb} \cite{hogan2012multiple}. 

Following such a procedure, we can use \Cref{thm:LME} to extend Definition \ref{def:lme} to relate to any optimal set of primal and dual variables for the combined model, leading to 
\begin{definition}[Extended LMEs]\label{def:extlme}
For a fixed pair of optimal primal and dual solutions to the \ref{p:OPF}, the LME at each node is given by
\begin{equation} 
    LME_i:=p_{\pi,i}+p_o\pi_i.
\end{equation}
\end{definition}
Just as in the case of locational marginal pricing, in \Cref{sec:dec-eq} we show that such a definition of LMEs retains the same market clearing and additivity properties without the assumption of differentiability.

\subsection{Carbon Accounting with LMEs} \label{sec:accounting}

The ability of LMEs to accurately evaluate marginal actions makes them well positioned for use in an accounting scheme which encompasses both consumer allocations and market-based mechanisms. We develop LME-based accounting through considering each grid dispatch property to be an independent action and using marginal rates to evaluate the effect of such actions. We inform our approach using classical locational marginal pricing accounts and motivate the emissions ``mirrored'' version to reflect the economic approach.

\begin{remark}[Abuse of Notation]
    In the following sections it is useful to reference properties of the nodes at which generators are situated to define generator accounts and properties. To do so conveniently we define a nodal property subscripted by a generator to refer to the property of the node at which the generator is located, for example: $LME_g := LME_i \mid g \in G_i$.
\end{remark} 

\paragraph{Load: } Under locational marginal pricing, grid operators evaluate the total consumption account at its marginal rate, multiplying the locational marginal price by the load size to determine the total load account. Similarly, for LMEs we evaluate the total load at its marginal emission rate:
\begin{equation}
    LME_iP^D_i \label{eq:acc:load}.
\end{equation}
\paragraph{Generator: } 
Grid operators compensate generators at their node's marginal rate (LMP multiplied by generation quantity). However, the total economic account must also factor in production costs. In the LME framework, this manifests as generators receiving a credit (negative cost) for total generation evaluated at the marginal emissions rate, offset by the direct emissions cost associated with production:
\begin{equation}
    (\sigma_g-LME_g)P^G_g \label{eq:acc:gen}.
\end{equation}

\paragraph{Transmission: } 
The final grid entity is that of transmission. In marginal pricing, each transmission line has an associated quantity known as ''congestion rent'' which reflects the shadow price associated with a given line's flow limit. The associated transmission economic accounts are then defined by evaluating a line's total flow at its congestion rent rate. 

Thus, to analyze transmission resources we introduce the notion of shadow carbon intensity, or $SCI_\ell$, which describes how the system emissions would change with respect to an increase in a line's power limit. Combining our formulation of the \ref{p:OPF} and \Cref{thm:freund} yields: 
\begin{equation}
    SCI_\ell := p^+_{\rho,\ell}-p^-_{\rho,\ell}+p_o(\rho^+_\ell-\rho^-_\ell)
\end{equation} where $p^-_{\rho,\ell}, p^+_{\rho,\ell}$ refer to the dual value of \eqref{eq:cm:fm} and \eqref{eq:cm:fp} at line $\ell$ respectively. Evaluating the total magnitude of line flow at this rate leads to
\begin{equation}
    SCI_\ell |f_\ell| \label{eq:acc:line}.
\end{equation} 

Note that any component's carbon account can be both positive or negative depending on the component's immediate effect on the system. Renewable or low carbon generation resources typically posses negative carbon accounts as the nodal LME will likely be greater than their generation carbon intensity, while dirtier base-load generation such as coal will likely posses positive carbon accounts. It should also be noted that when renewables make up a part of the marginal generation source, as happens in periods of high congestion and curtailment, some nodes will have LMEs of 0. Loads typically have positive carbon accounts, although loads can be located at nodes which exhibit zero or negative valued LMEs if renewables are on the margin or if line congestion is such that increasing consumption as a node would enable the dispatch to shift towards lower carbon resources. Finally, transmission accounts will typically be zero valued as most lines should be generally uncongested, however in the case of congestion, transmission accounts can take either positive or negative values depending on if the congestion is limiting dirty or clean dispatch. 

Under such an accounting scheme, consumers with positive carbon accounts may seek to lower their emissions responsibility through the purchase of negative credits from other accounts such as renewable generation resources, or through undertaking projects which will generate negative carbon accounts such as building new renewable generation in high LME areas or increasing the transmission capacity of transmission lines with highly negative shadow carbon intensities.

\begin{remark}[Developing Intuition]
    For further intuition behind the meaning of LMEs and their associated carbon accounts, we provide a set of simplified examples in \Cref{sec:intuition} which describe the implementation of such an LME accounting scheme and provide insight behind the account breakdown and meanings.
\end{remark}

\section{The Fundamental Theorems of Two-Layered Mechanisms}\label{sec:dec-eq}

In this section, we examine the compatibility between the centralized optimization performed by the system operator and the decentralized profit-seeking behavior of individual market participants. In the context of modern power system markets, locational marginal prices (LMPs) serve as the coordinating signal which enables this support: they align the generator level objective of maximizing generator profit with the operator level objective of minimizing total system cost. We show that the LME-based carbon accounts align with the LMP-based economic accounting signals, enabling continued decentralized support for operator level objectives.
Finally, we show that markets resulting from LME-based carbon accounts accurately capture total system emissions, validating that the carbon footprint theorem holds under our proposed definition and calculation of LMEs.



\subsection{The Decentralized Objective: Defining ``Carbon Profit''}
One can consider the formulation of the \ref{p:OPF} as satisfying the \emph{secondary} objective of emissions management while fully enforcing the \emph{primary} objective of economic efficiency. Similarly, we can consider a two-tiered decentralized problem where each generator seeks to satisfy the \emph{primary} objective of minimizing its traditional economic cost account, i.e.,\footnote{where $LMP_i=\pi_i$ is the marginal price of increasing demand at node $i$ in the \ref{p:dcOPF}}
\begin{equation}
    P^{G*}_g := \text{argmin}_{P^{min}_g \leq P^G_g \leq P^{max}_g} (c_g-LMP_g)P^G_g,  \label{eq:gencost-dec}
\end{equation}
and additionally, within the set of dispatch solutions satisfying economic optimality of the first objective, the generators seek to minimize the \emph{secondary} objective of their carbon accounts:
\begin{equation}
    \min_{P^G_g \in P^{G*}_g}(\sigma_g-LME_g) P^G_g. \label{eq:gen-dec} 
\end{equation}

\begin{remark}[Aggregate Notation]
    In the following analyses, it is useful to refer to an aggregate of optimal solutions to the decentralized equations. Let $P^G$ denote a vector where each component $P^G_g$ is an optimal solution to the decentralized problem of the $gth$ generator. Furthermore, any complete vector of generation values summing in total to the overall demand can be shown to have a unique ``induced'' vector of $f$ and $\theta$ values which create a valid electric flow respecting \eqref{eq:dcopf:fb} and \eqref{eq:dcopf:angle}. We refer to these as the induced flows and voltage angles corresponding to a generation vector.
\end{remark}

\subsection{The First Theorem (Sufficiency): Decentralized Carbon Profit Maximization Leads to Minimized System Emissions} 
\begin{theorem} \label{thm:fft}
    Let $\hat{LME}=\hat{p}_{\pi}+\hat{p}_o\hat{\pi}$ and $\hat{LMP}=\hat{\pi}$ be LME and LMP values derived from an arbitrary optimal pair of primal and dual solutions to the \ref{p:OPF} denoted by $\hat{P}^G$,$\hat{f}$, $\hat{\theta}$, $\hat{\pi}$, $\hat{\gamma}^+$, $\hat{\gamma}^-$,$\hat{\rho}^+$, $\hat{\rho}^-$, $\hat{z}$, and $\hat{p}$. Any vector, $P^{G}$, of optimal solutions to the decentralized generation problem for each generator \eqref{eq:gen-dec} given pricing parameters $\hat{LME}$ and $\hat{LMP}$ whose induced flows and voltage angles, $f$ and $\theta$, satisfies transmission \eqref{eq:dcopf:angle} 
    \eqref{eq:dcopf:fmax}
    \eqref{eq:dcopf:fmin} and flow balance \eqref{eq:dcopf:fb} constraints defines an assignment of decision variables in an optimal solution to \ref{p:OPF}.
\end{theorem}
\begin{proof}

    We begin by making explicit a few useful constraints which must be satisfied by any feasible Lagrangian dual solution to the \ref{p:OPF}.
    \begin{align}
        &\hat{p}_{\pi,i}-\hat{p}_{\pi,j}-\hat{p}_{z,\ell} - \hat{p}^+_{\rho,\ell} - \hat{p}^-_{\rho,\ell} =0 ,&\forall \ell=(i,j)\in L, &~[f_\ell], \label{eq:dual:f}\\
        &\sum_{\ell = (i,j)}B_{\ell} \hat{p}_{z,\ell}-\sum_{\ell = (j,i)}B_{\ell}  \hat{p}_{z,\ell} =0 ,&\forall i \in N, &~[\theta_i], \label{eq:dual:voltage} \\
        &\sigma_g-\hat{p}_{\pi,g}-\hat{p}^+_{\gamma,g}-\hat{p}^-_{\gamma,g}-\hat{p}_o c_g=0, &\forall g \in G, &~[P^G_g]\label{eq:dual:pg}\\
        &\hat{p}^+_{\gamma}, 
        \hat{p}^+_{\rho} \leq 0, \label{eq:dual:plus}\\
        &\hat{p}^-_{\gamma}, 
        \hat{p}^-_{\rho} \geq 0.\label{eq:dual:minus}
    \end{align}
    
    We proceed with proof by contradiction. For the sake of contradiction, suppose \Cref{thm:fft} is false and $P^G$, $f$, $\theta$ is not an assignment of decision variables in any optimal solution to the \ref{p:OPF}. 
    
    \Cref{thm:lmp:sw} implies that $P^G, f, \theta$ is an optimal solution to the \ref{p:dcOPF} and thus in the feasible set of the \ref{p:OPF}. This feasibility implies it must be the case that $P^G$ is a less than optimal solution to the \ref{p:OPF}, i.e. $\sigma^T P^G  > \sigma^T \hat{P}^G$.  Consider $\Delta P^G = P^G-\hat{P}^G$ and note that from \Cref{thm:lmp:sw} it must be that $\hat{P}^G$ is a valid optimal generation set for \eqref{eq:gencost-dec} and thus both $P^G_g$ and $\hat{P}^G_g$ lie in the convex solution space of $P^{G*}_g$. The optimality conditions for \eqref{eq:gen-dec} and \eqref{eq:gencost-dec} lead to the following implications:
    \begin{align}
        \Delta P^G_i > 0 \implies \sigma_i \leq \hat{LME}_i = \hat{p}_{\pi,i}+\hat{p}_o\hat{\pi}_i = \hat{p}_{\pi,i}+\hat{p}_oc_i, \\
       \Delta P^G_i < 0 \implies \sigma_i \geq \hat{LME}_i = \hat{p}_{\pi,i}+\hat{p}_o\hat{\pi}_i = \hat{p}_{\pi,i}+\hat{p}_oc_i .
    \end{align}
     
    Combining to yield the inequality relation
    \begin{equation}
        \sigma^T P^G - \sigma^T \hat{P}^G = \sigma^T \Delta P^G \leq (\hat{p}_{\pi}+\hat{p}_o c)^T \Delta P^G = \hat{p}_{\pi}^T \Delta P^G.
    \end{equation}

    Now, consider the net line flow difference $\Delta f = f-\hat{f}$. Note that from the feasibility of both solutions to the \ref{p:dcOPF} and complementary slackness of \eqref{eq:cm:fp} and \eqref{eq:cm:fm}, we observe the following implications:
    \begin{align}
        \Delta f_\ell > 0 \implies \hat{p}^+_{\rho,\ell} = 0,\\
        \Delta f_\ell < 0 \implies \hat{p}^-_{\rho,\ell} = 0.
    \end{align}
    Combining the above implications and applying \eqref{eq:dual:f} where $\ell=(i,j)$ yields
    \begin{equation}
        \Delta f_\ell (\hat{p}_{\pi,j} - \hat{p}_{\pi,i}) = -\Delta f_\ell (\hat{p}_{z,\ell} + \hat{p}^+_{\rho,\ell} + \hat{p}^-_{\rho,\ell}) \leq -\Delta f_\ell \hat{p}_{z,\ell}. \label{eq:fft_bound}
    \end{equation}
 Finally, define $\Delta \theta = \theta-\hat{\theta}$. Using feasibility of $P^G$ and $\hat{P}^G$ to \eqref{eq:dcopf:angle} and \eqref{eq:dcopf:fb} along with \eqref{eq:dual:voltage}, we apply the above inequality to observe
    \begin{align}
       \sum_{i \in G} \Delta P^G_i \hat{p}_{\pi,i} &= \sum_{\ell=(i,j) \in L} \Delta f_\ell (\hat{p}_{\pi,j} - \hat{p}_{\pi,i}) \\
        &\leq -\sum_{\ell=(i,j) \in L} \Delta f_\ell \hat{p}_{z,\ell} \\
        &= -\sum_{\ell=(i,j) \in L} B_{i,j} (\Delta\theta_i-\Delta\theta_j) \hat{p}_{z,\ell} \\
        &= -\sum_{i \in N}(\Delta\theta_i)( \sum_{\ell = (i,j)}B_{\ell} \hat{p}_{z,\ell}-\sum_{\ell = (j,i)}B_{\ell}  \hat{p}_{z,\ell}) \\
        &= 0.
    \end{align}
However, from \eqref{eq:fft_bound} this contradicts the assumption that $\sigma^T P^G >\sigma^T \hat{P}^G$.
\end{proof}
\subsection{The Second Theorem (Necessity): Existence of the LME Signal}
\begin{theorem} \label{thm:sft}
    Let $\hat{LME}=\hat{p}_{\pi}+\hat{p}_o\hat{\pi}$ and $\hat{LMP}=\hat{\pi}$ be LME and LMP values derived from an arbitrary optimal pair of primal and dual solutions to the \ref{p:OPF} denoted by $\hat{P}^G$,$\hat{f}$, $\hat{\theta}$, $\hat{\pi}$, $\hat{\gamma}^+$, $\hat{\gamma}^-$,$\hat{\rho}^+$, $\hat{\rho}^-$, $\hat{z}$, and $\hat{p}$. Given pricing parameters $\hat{LME}$ and $\hat{LMP}$, $\hat{P}^G_g$ is an optimal solution to \eqref{eq:gen-dec} for any generator, $g$.
\end{theorem}
\begin{proof}
    We begin by observing that from \Cref{thm:lmp:sw} it must be the case that $\hat{P}^G_g$ is an optimal solution to \eqref{eq:gencost-dec}, i.e. $\hat{P}^G_g \in P^{G*}_g$ and thus it remains to be shown that $\hat{P}^G_g$ satisfies the optimality conditions for \eqref{eq:gen-dec}.
    Note that if dim($P^{G*}_g$) $=0$, i.e., if \eqref{eq:gencost-dec} supports only a point solution $\hat{P}^G_g$ will be trivially optimal for \eqref{eq:gen-dec}, thus we proceed assuming dim($P^{G*}_g$) $>0$ and \eqref{eq:gencost-dec} supports multiple solutions. However, for this to be the case, it must be the case that $c_g-LMP_g=0$ as otherwise there would only be one valid solution to \eqref{eq:gencost-dec} ($P^{max}_g$ or $P^{min}_g$ depending on the sign of $c_g-LMP_g$). 
    Therefore, we see that $c_g$ is precisely $LMP_g=\pi_g$ and from \eqref{eq:dual:pg} 
    \begin{equation}
        \sigma_g = \hat{p}_{\pi,g}+\hat{p}_o \hat{\pi}_g+\hat{p}^+_{\gamma,g}+\hat{p}^-_{\gamma,g}=\hat{LME_g}+\hat{p}^+_{\gamma,g}+\hat{p}^-_{\gamma,g}.
    \end{equation}
    Finally, combining \eqref{eq:dual:plus} and \eqref{eq:dual:minus} with complementary slackness of the \ref{p:OPF} requires that $\hat{p}^+_{\gamma,g} \not = 0 \implies \hat{P}^G_g=P^{max}_g$ and $\hat{p}^-_{\gamma,g} \not = 0 \implies \hat{P}^G_g=P^{min}_g$ yields the following implications:
    \begin{equation}
        \sigma_g < \hat{LME}_g \implies \hat{p}^+_{\gamma,g} \not = 0 \implies \hat{P}^G_g=P^{max}_g,
    \end{equation}
     \begin{equation}
        \sigma_g > \hat{LME}_g \implies \hat{p}^-_{\gamma,g} \not = 0 \implies \hat{P}^G_g=P^{min}_g.
    \end{equation}
    
\end{proof}
\subsection{The Carbon Footprint Theorem}\label{sec:cft}
Just as the financial accounts across all components of the power grid sum to the total system cost under LMP based pricing, strong duality of linear programs shows that the carbon accounts sum to the total carbon emissions. Such a result has been previously shown in \cite{ruiz2010analysis} for the redispatch based definition of LMEs. Here we extend the result to the introduced Definition \ref{def:extlme}, analyzing additivity under the two-level market mechanisms.
\begin{theorem}[Carbon Footprint Theorem] \label{thm:carbonfootprint}
The carbon accounts of all loads, generators, and transmission lines sum to the total system carbon emissions:
\begin{equation} \label{eq:carbonfootprint}
    \sum_{g \in G} \sigma_g P^G_g = \sum_{i \in N}LME_i P^D_i + \sum_{g \in G}(\sigma_g - LME_g) P^G_g + \sum_{\ell \in L} SCI_\ell |f_\ell|.
\end{equation}
\end{theorem}
\begin{proof}{of \Cref{thm:carbonfootprint}}
    We begin by defining an arbitrary set of optimal primal and dual solutions to the \ref{p:OPF} of the form $P^G$,$f$, $\theta$, $\pi$, $\gamma^+$, $\gamma^-$,$\rho^+$, $\rho^-$, $z$, and $p$.
    We now proceed to evaluate the sum of carbon accounts:
    \begin{equation}
        \sum_{i \in N}LME_i P^D_i + \sum_{g \in G}(\sigma_g - LME_g) P^G_g + \sum_{\ell \in L} SCI_\ell |f_\ell|.
    \end{equation}
    Plugging in definition of $LME_i$ and $SCI_i$ gives
    \begin{equation}
        \sum_{i \in N}\left((p_{\pi,i}+p_o\pi_i) P^D_i\right) + \sum_{g \in G}\left((\sigma_g - p_{\pi,g}-p_o\pi_g) P^G_g\right) + \sum_{\ell \in L} \left( (p^+_{\rho,\ell}-p^-_{\rho,\ell}+p_o(\rho_\ell^+-\rho_\ell^-))|f_\ell|\right) 
    \end{equation}
    which from complementary slackness between $f_\ell$ and its corresponding dual variables for both the \ref{p:OPF} and the \ref{p:dcOPF} leads to
    \begin{equation}
        \sum_{i \in N}\left((p_{\pi,i}+p_o\pi_i) P^D_i\right) + \sum_{g \in G}\left((\sigma_g - p_{\pi,g}-p_o\pi_g) P^G_g\right) + \sum_{\ell \in L} \left( (p^+_{\rho,\ell}+p^-_{\rho,\ell}+p_o(\rho_\ell^++\rho_\ell^-))f_\ell \right) 
    \end{equation}
    and using strong duality of the \ref{p:dcOPF} simplifies to
    \begin{equation}
       \sum_{i \in N}\left(p_{\pi,i} P^D_i\right) + \sum_{g \in G}\left((\sigma_g - p_{\pi,g}) P^G_g\right) + \sum_{\ell \in L} \left( (p^+_{\rho,\ell}+p^-_{\rho,\ell})f_\ell \right). 
    \end{equation}
    Substituting $(p^+_{\rho,\ell}+p^-_{\rho,\ell})$ using \eqref{eq:dual:f} and rearranging leads to
    \begin{equation}
        \sum_{i \in N}\left(p_{\pi,i} (P^D_i-\sum_{g \in G_i}P^G_g+\sum_{\ell=(i,j)\in L}f_\ell-\sum_{\ell=(j,i)\in L}f_\ell)\right) + \sum_{g \in G}\left(\sigma_g P^G_g\right) - \sum_{\ell=(i,j) \in L} \left(p_{z,\ell}f_\ell \right).
    \end{equation}
    Where plugging in \eqref{eq:cm:kvl} and \eqref{eq:cm:fb} simplifies to
    \begin{equation}
        \sum_{g \in G}\left(\sigma_g P^G_g\right) - \sum_{i\in N}\theta_i \left(\sum_{\ell=(i,j)\in L}p_{z,\ell} B_{\ell}-\sum_{\ell=(j,i)\in L}p_{z,\ell}B_{\ell} \right)
    \end{equation}
    which using \eqref{eq:dual:voltage} simplifies to
    \begin{equation}
        \sum_{g \in G}\sigma_g P^G_g.
    \end{equation}
\end{proof}

\subsection{Importance of LME Accounting Properties}\label{sec:import}
As previously mentioned, marginal accounting schemes in the form of locational marginal pricing have seen significant adoption in electricity markets. The driving forces behind such adoption have been the benefits behind the following two properties:
\begin{enumerate}
    \item \emph{Decentralized equilibrium}, the use of LMPs incentivizes generation resources in a way consistent with the system operator's central dispatch. This significantly decreases the need for costly forcing mechanisms such as generator on/off payments or non-cooperation fines.
    \item \emph{Total cost consistency}, the net sum of financial accounts across generation, load, and transmission under LMP pricing is known to sum to precisely the total system cost. This ensures the net cost is properly covered by the system entities without additional cost externalities or deficits.
\end{enumerate}

Under LME accounting we observed that such key properties are preserved through \Cref{thm:fft}, \Cref{thm:sft}, and \Cref{thm:carbonfootprint}. In particular, such properties support a two-stage dispatch which first minimizes economic costs, then, while remaining economically optimal, seeks to minimize emissions. Here, the primary problem ensures the accounting scheme is aligned with the current economic dispatch method of system operation, while the secondary objective allows for the introduction of carbon impacts to the account ledger and furthermore, incentivizes the lowest possible carbon dispatch while maintaining normal (economic) system operation. In \Cref{sec:texas} we explore the empirical benefits of these theorems and other key LME properties.

\subsection{Practical Examples and Intuition of LME Accounting}\label{sec:intuition}
In order to familiarize the reader with the implementation of such an LME accounting scheme and provide insight behind the account breakdown and meanings we submit this section which considers a set of simplified grid examples.

\subsubsection{Transmission Constrained Example}
\begin{figure}[h]
    \centering
    \resizebox{0.4\textwidth}{!}{
        \begin{tikzpicture}[every text node part/.style={align=center},
bus/.style={circle, draw=gray!60, fill=gray!5, very thick, minimum size=30mm},
fossil/.style={regular polygon,regular polygon sides=4, draw=red!60, fill=red!5, very thick, minimum size=5mm},
renewable/.style={regular polygon,regular polygon sides=4, draw=green!60, fill=green!5, very thick, minimum size=5mm},
load/.style={regular polygon,regular polygon sides=5, draw=blue!60, fill=blue!5, very thick, minimum size=3mm},
]
\node[bus] at (-4, 0)     (bus1)                        {Bus 1\\ \textcolor{purple}{(0 unit CO2/MWh)}};
\node[bus] at (4, 0)      (bus2)                          {Bus 2\\ \textcolor{purple}{(1 unit CO2/MWh)}};
\node[load]     (load)       [above=of bus2] {Load 1\\(250 MW)};
\node[fossil]   (fossilgen)  [below=of bus2] [label={[text=red]below: Dispatch 150 MW}] {Gen 2 \\(200 MW)};
\node[renewable]   (renewgen)  [below=of bus1] [label={[text=red]below: Dispatch 100 MW}] {Gen 1\\(200 MW)};

\draw[->, very thick] (bus1.east) -- (bus2.west) node [midway, above, sloped] (line lim) {Capacity 100 MW} node [midway, below, sloped, red] (line flow) {Flow 100 MW};

\draw[dotted, very thick] (bus2.north) -- (load.south);
\draw[dotted, very thick] (bus1.south) -- (renewgen.north);
\draw[dotted, very thick] (bus2.south) -- (fossilgen.north); 
\end{tikzpicture}
        }
    \caption{Transmission constrained two-node example, generator capacities, load magnitudes, and line capacities are shown in black while generator production and line flows are shown in red and nodal LME values are shown in purple. Gen 1 is set to be a VRE generator with a emissions intensity of 0 unit CO2/MWh while Gen 2 is set to be a fossil fuel generator with emissions intensity 1 unit CO2/MWh. It is assumed the VRE generation unit will be dispatched ahead of the fossil fuel.}
    \label{fig:2bus:trans}
\end{figure}
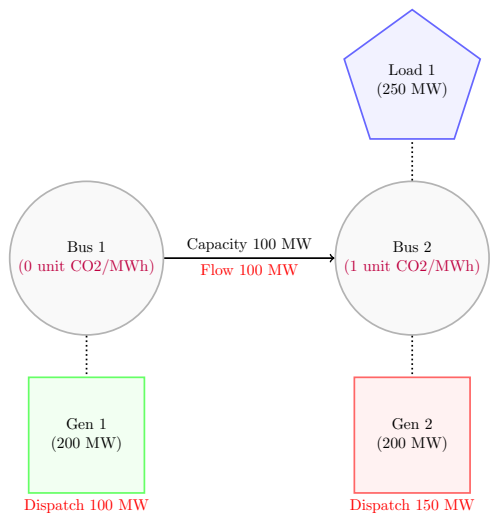
\Cref{fig:2bus:trans} displays a two node system where transmission capacity is limiting the dispatch of a renewable generator (Gen 1). Note that if a load is added at node 1 the increased demand can be met by generator 1, the renewable generator, as the generator is only dispatched 100 MW of its 200 MW total capacity, leading to an LME of 0 unit CO2/MWh. However, for added load at node 2, the increased demand will be met by generator 2 as the flow limit on the line connecting the nodes is already at capacity, leading to an LME of 1 unit CO2/MWh. Applying the carbon accounting scheme outlined in \Cref{sec:accounting} over an hour of dispatch yields a carbon account of 250 units CO2 for load 1, -100 units CO2 for the transmission infrastructure, and 0 units CO2 for all other entities.
\subsubsection{Generation Capacity Constrained Example}
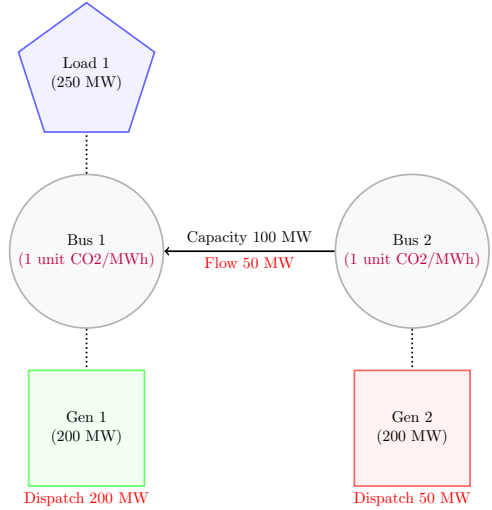
\begin{figure}[h]
    \centering
    \resizebox{0.4\textwidth}{!}{
        \begin{tikzpicture}[every text node part/.style={align=center},
bus/.style={circle, draw=gray!60, fill=gray!5, very thick, minimum size=30mm},
fossil/.style={regular polygon,regular polygon sides=4, draw=red!60, fill=red!5, very thick, minimum size=5mm},
renewable/.style={regular polygon,regular polygon sides=4, draw=green!60, fill=green!5, very thick, minimum size=5mm},
load/.style={regular polygon,regular polygon sides=5, draw=blue!60, fill=blue!5, very thick, minimum size=3mm},
]
\node[bus] at (-4, 0)     (bus1)                        {Bus 1\\ \textcolor{purple}{(1 unit CO2/MWh)}};
\node[bus] at (4, 0)      (bus2)                          {Bus 2\\ \textcolor{purple}{(1 unit CO2/MWh)}};
\node[load]     (load)       [above=of bus1] {Load 1\\(250 MW)};
\node[fossil]   (fossilgen)  [below=of bus2] [label={[text=red]below: Dispatch 50 MW}] {Gen 2 \\(200 MW)};
\node[renewable]   (renewgen)  [below=of bus1] [label={[text=red]below: Dispatch 200 MW}] {Gen 1\\(200 MW)};

\draw[<-, very thick] (bus1.east) -- (bus2.west) node [midway, above, sloped] (line lim) {Capacity 100 MW} node [midway, below, sloped, red] (line flow) {Flow 50 MW};

\draw[dotted, very thick] (bus1.north) -- (load.south);
\draw[dotted, very thick] (bus1.south) -- (renewgen.north);
\draw[dotted, very thick] (bus2.south) -- (fossilgen.north); 
\end{tikzpicture}
        }
    \caption{Renewable capacity constrained two-node example}
    \label{fig:2bus:gen}
\end{figure}
\Cref{fig:2bus:gen} displays a two node system where generation capacity is the only limiting factor in the dispatch of the renewable generator. Note that if a load is added at node 1 the increased demand must be met by generator 2 as generator 1 is already at capacity leading to an LME of 1 unit CO2/MWh. Similarly, for added load at node 2, the increased demand must also be met by generator 2, once again leading to an LME of 1 unit CO2/MWh. Indeed, generally for cases of non-transmission constrainted dispatch the marginal generator will be the same across all nodes. Now, applying the carbon accounting scheme outlined in \Cref{sec:accounting} over an hour of dispatch yields a carbon account of 250 units CO2 for load 1, -200 units CO2 for generator 1, and 0 units CO2 for all other entities.
\subsubsection{Carbon Account Investment Signals and Intuition}
If we consider carbon accounts to participate in a market-based scheme where positive consumers seek to lower their allocated totals and negative consumers may sell portions of their negative accounts, it is reasonable to view a negative account as a monetary injection, and a positive account as a cost. In \Cref{fig:2bus:trans} the negative carbon account, and associated monetary injection, are held by the transmission line rather than the renewable generation. While this may seem counterintuitive, we note that due to transmission constraints, the renewable generator is not able to dispatch to its full capacity, and thus providing capital for the upgrade of the transmission line over the renewable generator is desirable. 

In contrast, in \Cref{fig:2bus:gen} the negative carbon account is held by the renewable generation. In this case, the renewable generator dispatch is limited only by its capacity and thus providing capital for the upgrade of the renewable generator capacity is desirable while transmission upgrades would not enable a decrease in system emissions. Additionally, we see that in both cases the load receives a positive carbon account, inducing an additional cost on consumption. This is consistent as in each case, decreasing the total load would cause a decrease in system emissions, and thus the extra cost works as a signal to disincentivize excess/unnecessary consumption. 

Overall, in both cases we see the carbon account structures are tied to effective decarbonization strategies for the grid configuration.

\section{Inter-temporal Dispatch: Extending LMEs to Energy Storage}\label{sec:storage}
In practice, the \ref{p:dcOPF} may not always properly capture all the important constraints and requirements of grid operations and there exist many common variants to account for reliability constraints, multiperiod constraints, and more advanced cost/generation constraints. As described in previous sections \Cref{thm:LME}, \Cref{thm:fft}, \Cref{thm:sft}, and \Cref{thm:carbonfootprint} rely mostly on the linearity and subsequent strong duality of the \ref{p:dcOPF}, allowing for straightforward generalizations to linear variants of the \ref{p:dcOPF} and even linearizations of the AC-OPF model such as those in \cite{bienstock_accurate_2025}. 

In this section we explore one such generalization of \ref{p:dcOPF} which includes dispatchable storage resources with changing states over multiple dispatch periods. Here, the set of storage resources is denoted by the set $S$, with the state of charge of storage element $s\in S$ denoted as $E_{s}$, the charging/discharging rate denoted as $P^{S+}_s$ and $P^{S-}_s$ respectively, and the charging/discharging efficiency denoted as $\eta$. Finally, we append the subscript $t \in T$ to denote any element at a particular time step. This leads to the following modified economic dispatch formulation:
\begin{align}
   \min \quad  & \sum_{t \in T} \sum_{g \in G} c_g P^G_{g,t} \tag{DCOPF Storage Model} \label{p:dcOPFstorage}\\
    \text{s.t. }
    &B_{ij} (\theta_{i,t} - \theta_{j,t}) - f_{\ell,t} =  0, &\forall \ell=(i,j) \in L,\label{eq:stor:voltage}\\
    &\sum_{g \in G_i} P^G_{g,t}  - \sum_{\ell = (i,j) \in L}f_{\ell,t} + \sum_{\ell = (j,i) \in L}f_{\ell,t} - \sum_{s \in S_i} (P^{S+}_{s,t} -P^{S-}_{s,t}) = P_{i,t}^D, & \forall i \in N, \forall t \in T, \label{eq:stor:fb}\\
    &f_{\ell,t} \leq f^{max}_{\ell}, & \forall \ell \in L, \forall t \in T, \label{eq:stor:fp}  \\
    &f_{\ell,t} \geq -f^{max}_{\ell}, & \forall \ell \in L, \forall t \in T, \label{eq:stor:fm} \\
    &P^G_{g,t} \leq P^{max}_g, & \forall g \in G, \forall t \in T,\\
    &P^G_{g,t} \geq P^{min}_g, & \forall g \in G, \forall t \in T, \\
    &E_{s,t} \geq 0, & \forall s \in S, \forall t \in T, \\
    &E_{s,t} \leq E^{max}_s, & \forall s \in S, \forall t \in T, \\
    &E_{s,t}+\eta P^{S+}_{s,t} - \frac{1}{\eta}P^{S-}_{s,t}-E_{s,t+1} = 0, & \forall s \in S, \forall t \in T,  \\
    &P^{S+},P^{S-} \geq 0.
\end{align} 
Due to their similar nature to the previous section, all further explicit models and proofs can be found in \Cref{sec:variants} for the sake of brevity.
Now, following the same steps as \Cref{sec:model} we can reformulate to create the combined model formulation, described explicitly in the \ref{p:storageCOPF} of \Cref{sec:variants}. Where, once again applying \Cref{thm:freund} shows that \Cref{def:extlme} remains unchanged and
\begin{equation} \label{eq:lme-stor}
    LME_{i,t}=p_{\pi,i,t}+p_o \pi_{i,t}.
\end{equation}
With this definition, we continue to use the carbon accounts defined in \Cref{sec:accounting} for load, generation, and transmission, however, we must also define carbon accounts for storage. Using the same inspiration as \Cref{sec:accounting} we note that energy injection and withdrawal displaces emissions at a rate of $LME_{s}$ and thus receives an overall carbon account of $LME_s\cdot (P^{S+}_s-P^{S-}_s)$. This allows us to define the generalized version of the carbon footprint theorem,
\Cref{thm:carbonfootprint}, in the storage context:
\begin{theorem}
[Carbon Footprint Theorem for Storage Variant] 
The carbon accounts of all loads, generators, transmission lines, and storage sum to the total system carbon emissions:
\label{eq:carbonfootprintstorage}
\begin{align}
    \sum_{t \in T}\sum_{g \in G} \sigma_{g} P^G_{g,t} = \sum_{t\in T}\Bigg(&\sum_{i \in N}LME_{i,t} P^D_{i,t} + \sum_{g \in G}(\sigma_{g,t} - LME_{g,t}) P^G_{g,t} + \sum_{\ell \in L} SCI_{\ell,t} |f_{\ell,t}| \notag \\
    & +\sum_{s \in S} LME_{s,t} (P^{S+}_{s,t}-P^{S-}_{s,t}) \Bigg).
\end{align}
\end{theorem}
We can additionally consider the decentralized problems present in the storage variant. In this case, the carbon accounts for generation remain defined the same, and one can observe that the proofs presented in \Cref{sec:dec-eq} remain true for the decentralized generation resources. However, here the storage resources also possess their own decisions and decentralized problems. Here, the storage operator's primary problem can be seen as maximizing the net profit from energy arbitrage over all periods, explicitly:
\begin{align}
    P^{S*}_s=\text{argmin}_{P^{S+}_s, P^{S-}_s, E_s \geq 0} \quad &\sum_{t \in T} LMP_{s,t}\cdot (P^{S+}_{s,t} - P^{S-}_{s,t}) \label{eq:storcost-dec} \\
    \text{s.t. }
    &E_{s,t} \leq E^{max}_s, & \forall t \in T,\\
    &E_{s,t}+\eta P^{S+}_{s,t} - \frac{1}{\eta}P^{S-}_{s,t}-E_{s,t+1} = 0, &  \forall t \in T. 
\end{align}
Furthermore, within the set of storage dispatch satisfying economic optimality of the primary objective (referred to as $D_s^*$), we can define the secondary storage operations problem as 
\begin{align}
    \min_{(P^{S+}_s, P^{S-}_s) \in P^{S*}_s} \quad &\sum_{t \in T} LME_{s,t}\cdot (P^{S+}_{s,t} - P^{S-}_{s,t}) \label{eq:stor-dec} 
\end{align}
Despite the more complex storage decentralized problem spanning across multiple time periods we retain the same equilibrium support as seen in \Cref{thm:fft} and \Cref{thm:sft} for the \ref{p:dcOPF}.
\begin{theorem} \label{thm:stor-fft}
    Let $\hat{LME}=\hat{p}_{\pi}+\hat{p}_o\hat{\pi}$ and $\hat{LMP}=\hat{\pi}$ be LME and LMP values derived from an arbitrary optimal pair of primal and dual solutions to the \ref{p:storageCOPF}. 
    Any vector, $P^{G}$, of optimal solutions to the decentralized generation problem for each generator \eqref{eq:gen-dec} and $P^{S+},P^{S-}$ of optimal solutions to the decentralized generation problem for each storage element \eqref{eq:stor-dec} given pricing parameters $\hat{LME}$ and $\hat{LMP}$ whose induced flows and voltage angles, $f$ and $\theta$, satisfies transmission \eqref{eq:stor:voltage} 
    \eqref{eq:stor:fp}
    \eqref{eq:stor:fm} and flow balance \eqref{eq:stor:fb} constraints defines an assignment of decision variables in an optimal solution to the \ref{p:storageCOPF}.
\end{theorem}
\begin{theorem} \label{thm:stor-sft}
    Let $\hat{LME}=\hat{p}_{\pi}+\hat{p}_o\hat{\pi}$ and $\hat{LMP}=\hat{\pi}$ be LME and LMP values derived from an arbitrary optimal pair of primal and dual solutions to the \ref{p:storageCOPF}. The values $\hat{P}^{S+}_s$ and $\hat{P}^{S-}_s$ within the optimal primal solution to the \ref{p:storageCOPF} is an optimal solution to \eqref{eq:stor-dec} for any storage resource, $s$, given pricing parameters $\hat{LME}$ and $\hat{LMP}$.
\end{theorem}

\subsection{Examples and Intuition of LME Accounting with Storage}
In the proceeding \Cref{sec:storage-expts} we observe that the inclusion of storage can decrease temporal and spatial variations in LMEs. To conclude this section, we seek to build intuition around the mechanism by which the inclusion of storage realizes such an effect through considering a simple example.

\begin{figure}[h]
    \centering
    \resizebox{0.9\textwidth}{!}{

\usetikzlibrary{positioning, shapes.geometric}


\begin{tikzpicture}[every text node part/.style={align=center},
    bus/.style={circle, draw=gray!60, fill=gray!5, very thick, minimum size=30mm},
    fossil/.style={regular polygon,regular polygon sides=4, draw=red!60, fill=red!5, very thick, minimum size=5mm},
    renewable/.style={regular polygon,regular polygon sides=4, draw=green!60, fill=green!5, very thick, minimum size=5mm},
    load/.style={regular polygon,regular polygon sides=5, draw=blue!60, fill=blue!5, very thick, minimum size=3mm},
]

\begin{scope}[name prefix=d1-]
    \node[bus] at (-4, 0)      (bus1)      {Bus 1\\ \textcolor{purple}{(0 unit CO2/MWh)}};
    \node[bus] at (4, 0)       (bus2)      {Bus 2\\ \textcolor{purple}{(1 unit CO2/MWh)}};
    \node[load]               (load1)      [left=of bus1] {Load 1\\(50 MW)};
    \node[load]               (load2)      [right=of bus2] {Load 2\\(150 MW)};
    \node[fossil]             (fossilgen) [below=of bus2] [label={[text=red]below: Dispatch 50 MW}] {Gen 2 \\(300 MW)};
    \node[renewable]          (renewgen)  [below=of bus1] [label={[text=red]below: Dispatch 150 MW}] {Gen 1\\(200 MW)};

    \draw[->, very thick] (bus1.east) -- (bus2.west) node [midway, above, sloped] (line lim) {Capacity 100 MW} node [midway, below, sloped, red] (line flow) {Flow 100 MW};
    \draw[dotted, very thick] (bus1.west) -- (load1.east);
    \draw[dotted, very thick] (bus2.east) -- (load2.west);
    \draw[dotted, very thick] (bus1.south) -- (renewgen.north);
    \draw[dotted, very thick] (bus2.south) -- (fossilgen.north);
\end{scope}

\begin{scope}[shift={(22cm, 0)}, name prefix=d2-]
    
    \node[bus] at (-4, 0)      (bus1)      {Bus 1\\ \textcolor{purple}{(1 unit CO2/MWh)}};
    \node[bus] at (4, 0)       (bus2)      {Bus 2\\ \textcolor{purple}{(1 unit CO2/MWh)}};
    \node[load]               (load1)      [left=of bus1] {Load 1\\(50 MW)};
    \node[load]               (load2)      [right=of bus2] {Load 2\\(150 MW)};
    \node[fossil]             (fossilgen) [below=of bus2] [label={[text=red]below: Dispatch 200 MW}] {Gen 2 \\(300 MW)};
    \node[renewable]          (renewgen)  [below=of bus1] [label={[text=red]below: Dispatch 0 MW}] {Gen 1\\~~(0 MW)~~};

    \draw[<-, very thick] (bus1.east) -- (bus2.west) node [midway, above, sloped] (line lim) {Capacity 100 MW} node [midway, below, sloped, red] (line flow) {Flow 50 MW};
    \draw[dotted, very thick] (bus1.west) -- (load1.east);
    \draw[dotted, very thick] (bus2.east) -- (load2.west);
    \draw[dotted, very thick] (bus1.south) -- (renewgen.north);
    \draw[dotted, very thick] (bus2.south) -- (fossilgen.north);
\end{scope}

\node[font=\large] at (0, -6.5) {Dispatch Period 1};
\node[font=\large] at (22, -6.5) {Dispatch Period 2};

\end{tikzpicture}

        }
    \caption{Baseline two dispatch period two-node example, network properties are described in the same manner as \Cref{fig:2bus:trans}. The figure describes the network at two distinct dispatch periods, one high renewable dispatch period (left) and one low renewable dispatch period (right).}
    \label{fig:2bus:nostor}
\end{figure}
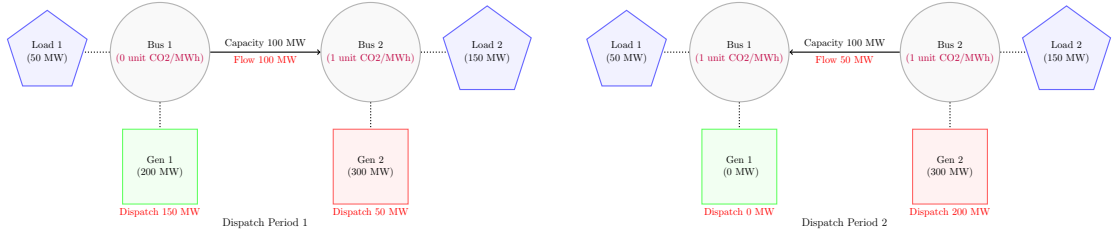

\Cref{fig:2bus:nostor} illustrates a two-node system across two distinct dispatch periods, the first characterized by high renewable availability and the second by no renewable availability. In the first period, transmission constraints prevent the full dispatch of the renewable generator and, consequently, additional load at node 1 is met by renewable generation (yielding an LME of 0 units), while additional load at node 2 requires fossil fuel generation (yielding an LME of 1 unit). In the second period, renewable generation is unavailable; thus, all marginal demand must be met by the fossil fuel generator.

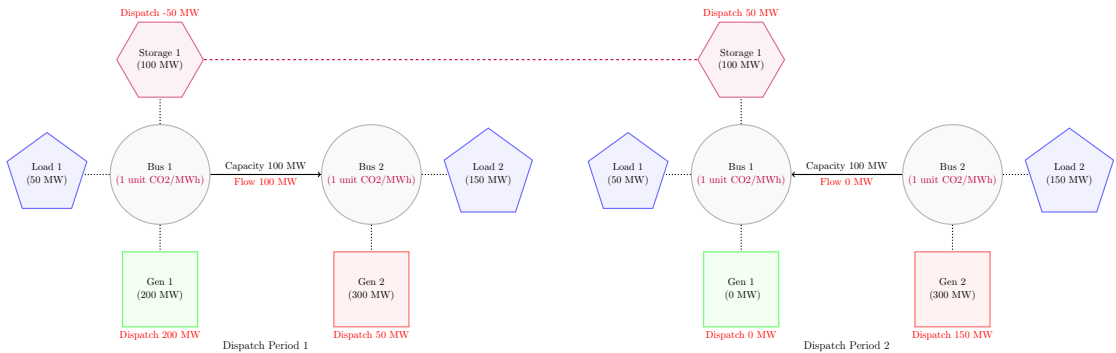
\begin{figure}[h]
    \centering
    \resizebox{0.9\textwidth}{!}{



\begin{tikzpicture}[every text node part/.style={align=center},
    bus/.style={circle, draw=gray!60, fill=gray!5, very thick, minimum size=30mm},
    fossil/.style={regular polygon,regular polygon sides=4, draw=red!60, fill=red!5, very thick, minimum size=5mm},
    renewable/.style={regular polygon,regular polygon sides=4, draw=green!60, fill=green!5, very thick, minimum size=5mm},
    load/.style={regular polygon,regular polygon sides=5, draw=blue!60, fill=blue!5, very thick, minimum size=3mm},
    storage/.style={regular polygon, regular polygon sides=6, draw=purple!60, fill=purple!5, very thick, minimum size=5mm}, 
]

\begin{scope}[name prefix=d1-]
    \node[bus] at (-4, 0)      (bus1)      {Bus 1\\ \textcolor{purple}{(1 unit CO2/MWh)}};
    \node[bus] at (4, 0)       (bus2)      {Bus 2\\ \textcolor{purple}{(1 unit CO2/MWh)}};
    \node[load]               (load1)     [left=of bus1] {Load 1\\(50 MW)};
    \node[load]               (load2)     [right=of bus2] {Load 2\\(150 MW)};
    \node[fossil]             (fossilgen) [below=of bus2] [label={[text=red]below: Dispatch 50 MW}] {Gen 2 \\(300 MW)};
    \node[renewable]          (renewgen)  [below=of bus1] [label={[text=red]below: Dispatch 200 MW}] {Gen 1\\(200 MW)};

    \node[storage] (storage) [above=of bus1] [label={[text=red]above: Dispatch -50 MW}]  {Storage 1\\(100 MW)};
    \draw[dotted, very thick] (bus1.north) -- (storage.south);

    \draw[->, very thick] (bus1.east) -- (bus2.west) node [midway, above, sloped] (line lim) {Capacity 100 MW} node [midway, below, sloped, red] (line flow) {Flow 100 MW};
    \draw[dotted, very thick] (bus1.west) -- (load1.east);
    \draw[dotted, very thick] (bus2.east) -- (load2.west);
    \draw[dotted, very thick] (bus1.south) -- (renewgen.north);
    \draw[dotted, very thick] (bus2.south) -- (fossilgen.north);
\end{scope}

\begin{scope}[shift={(22cm, 0)}, name prefix=d2-]
    
    \node[bus] at (-4, 0)      (bus1)      {Bus 1\\ \textcolor{purple}{(1 unit CO2/MWh)}};
    \node[bus] at (4, 0)       (bus2)      {Bus 2\\ \textcolor{purple}{(1 unit CO2/MWh)}};
    \node[load]               (load1)     [left=of bus1] {Load 1\\(50 MW)};
    \node[load]               (load2)     [right=of bus2] {Load 2\\(150 MW)};
    \node[fossil]             (fossilgen) [below=of bus2] [label={[text=red]below: Dispatch 150 MW}] {Gen 2 \\(300 MW)};
    \node[renewable]          (renewgen)  [below=of bus1] [label={[text=red]below: Dispatch 0 MW}] {Gen 1\\~~(0 MW)~~};
    
    \node[storage] (storage) [above=of bus1] [label={[text=red]above: Dispatch 50 MW}] {Storage 1\\(100 MW)};
    \draw[dotted, very thick] (bus1.north) -- (storage.south);

    \draw[<-, very thick] (bus1.east) -- (bus2.west) node [midway, above, sloped] (line lim) {Capacity 100 MW} node [midway, below, sloped, red] (line flow) {Flow 0 MW};
    \draw[dotted, very thick] (bus1.west) -- (load1.east);
    \draw[dotted, very thick] (bus2.east) -- (load2.west);
    \draw[dotted, very thick] (bus1.south) -- (renewgen.north);
    \draw[dotted, very thick] (bus2.south) -- (fossilgen.north);
\end{scope}

\draw[dashed, very thick, purple] (d1-storage) -> (d2-storage);

\node[font=\large] at (0, -6.5) {Dispatch Period 1};
\node[font=\large] at (22, -6.5) {Dispatch Period 2};

\end{tikzpicture}

        }
    \caption{Two dispatch period two-node example with storage. Storage elements are shown as purple hexagons and are considered to have a 100MW power capacity with 100\% efficiency and unlimited energy storage capacity.}
    \label{fig:2bus:stor}
\end{figure}
In \Cref{fig:2bus:stor} we consider how the addition of a storage element at node 1 changes the scenario in \Cref{fig:2bus:nostor}. Here, the storage element acts as a ``temporal'' directional transmission line, allowing the excess renewable production in dispatch period 1 to be consumed in dispatch period 2\footnote{In \Cref{fig:2bus:stor} we model the storage element as having a lossless efficiency and unbounded storage capacity but this can be generalized as a lossy transmission line between dispatch periods with a capacity equivalent to the battery energy capacity.}. This alleviates binding congestion in the first dispatch period (an equally optimal dispatch solution would be to send 51MW to the storage element and 99MW over the line from node 1 to node 2 in the first dispatch period), causing the fossil fuel generator to respond to additional load at all nodes in both dispatch periods.

\section{ERCOT Case Study}
\label{sec:texas}
In this section, we use a realistic representation of the modern ERCOT grid to investigate empirical properties of the described LME calculation and accounting methodology. To create a realistic representation of the ERCOT grid, we follow the methodology outlined in \cite{thomas_paper}. This leads to a grid with the characteristics described in \Cref{tab:ercot_data}. Additionally, we use historic load and weather data, once again generated in the manner of \cite{thomas_paper} to define a representative year of ERCOT hourly dispatch periods, each period defined by scaling factors for nodal loads and VRE capacity factors. The weather data describes hourly conditions at 20km resolution 2016 and is sourced from \cite{noaa_rap252}, while load data is sourced from \cite{EIA930}, describing hourly profiles across ERCOT.

\begin{table}
    \centering
    \begin{tabular}{ |c||c|c|c|c|c|c|c|c| } 
     \hline
      Entity & Node & Line & Load & Coal & Gas & Wind & Solar & Nuclear \\ 
        \hline
      Units (\#) & 385 & 466 & 202 & 11 & 83 & 106 & 36 & 2\\ 
      Max Capacity (GW) & N/A & 765.5 & 86.5 & 16.4 & 63.7 & 34.1 & 10.3 & 5.1\\ 
     \hline
    \end{tabular}
    \caption{Makeup of the simulated ERCOT grid}
    \label{tab:ercot_data}
\end{table}

\subsection{Characteristics}
\begin{figure}
   \begin{minipage}{0.48\textwidth}
     \centering
    \includegraphics[width=0.7\linewidth]{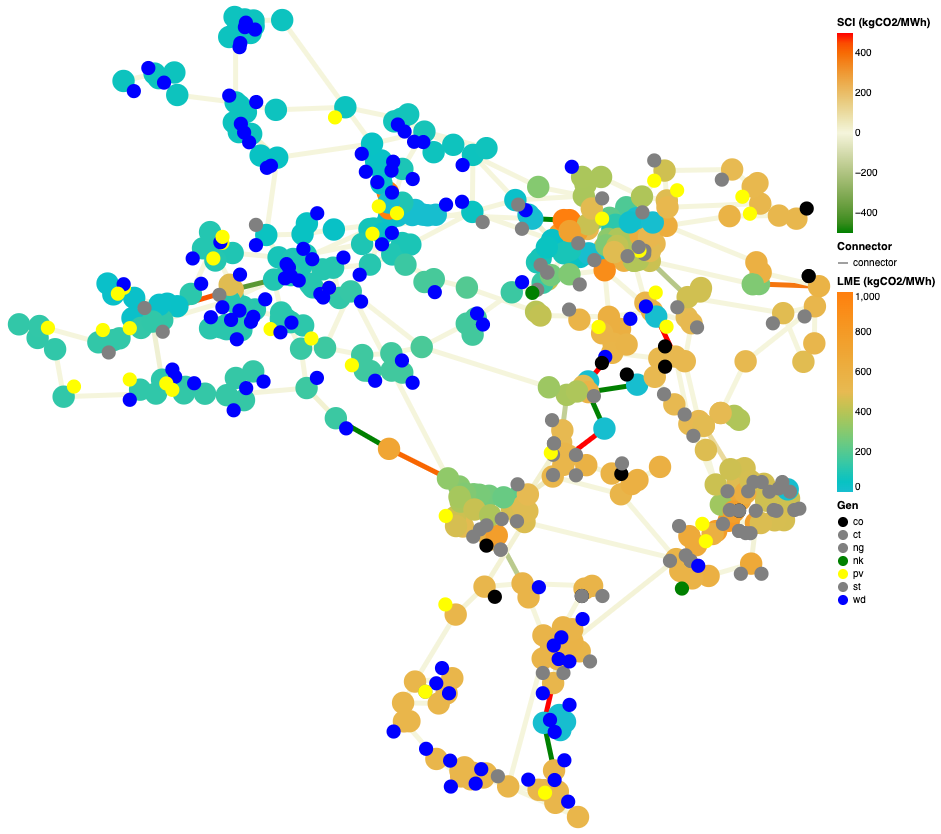}
    \caption{Average locational marginal emissions and shadow carbon intensities over the experimental horizon.}
    \label{fig:avgLME}
   \end{minipage}\hfill
   \begin{minipage}{0.48\textwidth}
     \centering
    \includegraphics[width=\linewidth]{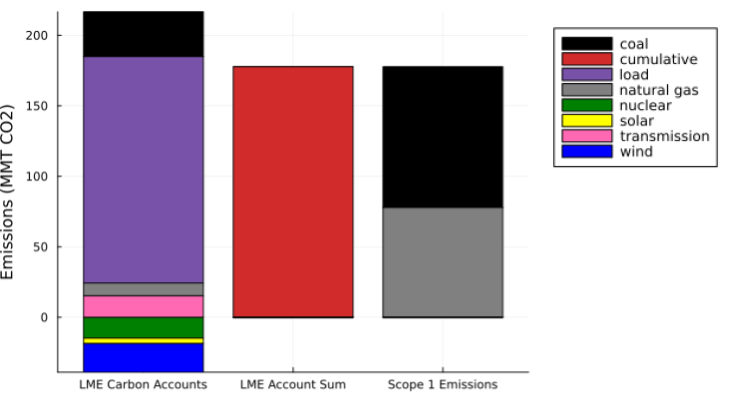}
    \caption{Total carbon accounts over the experimental horizon.}
    \label{fig:accounts}
   \end{minipage}
\end{figure}
The disparity between ERCOT's renewable generation and load locations is strongly reflected in the average nodal LMEs over the experimental horizon with nodes in the high wind and solar production northwest experiencing average values close to 200 kg CO2/MWh and nodes sited in lower production, higher demand eastern portions of the state experiencing average LME values closer to those of natural gas generators (450 kg CO2/MWh). Interestingly, while the south of the state does contain wind generation development, this portion of the state also exhibits some of the highest average LME factors, likely due to stronger transmission ties to the load centers of the eastern cities and a low enough total renewable capacity as to not operate as the marginal generation resource for significant periods of time. This indicates that decarbonization efforts may have close to double the carbon reduction impact per MWh if they focus renewable development in southern areas over their western counterparts. 

Such differences between western and eastern nodal LMEs also reflects a high potential for transmission upgrades to improve renewable deliverability within the Texas grid. Lines connecting the eastern load centers to western renewable generation displayed the highest magnitude carbon rents with one line spanning from Knoxville towards San Antonio and Austin having an average shadow carbon intensity of over -400 kgCO2/MWh. This indicates that in addition to aforementioned generation development in the southern portions of the state, transmission upgrades on east-west connection lines have high potential to have positive impacts on decarbonization through the alleviation of renewable-limiting congestion.
\subsection{Accounting}
Over the course of the experiment, we observed a total of 177.7 million metric tons of scope 1 carbon emissions from generation resources. Using the accounting framework outlined in \Cref{sec:accounting}, we observe that this translates into a carbon account of 160.6 MMT ascribed to the grid load, a net total of 1.8 MMT ascribed to generation resources, and a total of 15.2 MMT ascribed to transmission infrastructure. Removing the rounding done to make such numbers presentable, we get a total of 177.7 MMT of carbon emissions across all carbon accounts, directly matching the scope 1 emissions of the grid dispatch as described in \Cref{thm:carbonfootprint}. 
\begin{table}
    \centering
    \begin{tabular}{ |c||c|c|c|c|c| } 
     \hline
      Generation Type & Coal & Gas & Nuclear & Wind & Solar \\ 
        \hline
      Carbon Account (MMT CO2) & 31.7 & 9.01 & -14.8 & -20.4 & -3.8 \\ 
      Dispatch (GWh) & 115,493 & 172,953 & 34,662 & 100,605 & 14,552 \\ 
      Scope 1 Emissions (MMT CO2) & 99.9 & 77.8 & 0 & 0 & 0 \\
     \hline
    \end{tabular}
    \caption{Generation resource carbon accounts and total power production over the experimental horizon.}
    \label{tab:gen_accounts}
\end{table}

Upon closer inspection of the generation resource accounting, we see that coal generation experiences the highest carbon account with a total of 31.8 MMT CO2, while natural gas production received a much smaller total carbon account of 9.01 MMT CO2 despite having comparable amounts of scope 1 emissions. This is likely due to the tendency of natural gas to act as the marginal resource, while coal generation tends to be part of the baseload and is thus considered dirtier than the marginal resource a much higher portion of the time.

Additionally, we observe a distribution of positive and negative carbon accounts across transmission infrastructure. Lines whose congestion limits the dispatch of coal, necessitating an increased dispatch of natural gas, receive positive carbon accounts, indicating their upgrade would lead to an increase in emissions, while lines whose congestion limits the dispatch of renewable resources receive negative carbon accounts, indicating their upgrade would lead to a decrease in emissions. A natural assumption is that variable renewable generation resources are the most transmission constrained, so it is slightly surprising to see a significant number of both positive and negative transmission carbon accounts. However, the existence of both indicates the importance of a principled upgrade strategy and the utility that indicators such as the shadow carbon rent can provide for actors wishing to improve a transmission network in a method which encourages decarbonization.
\subsection{Patterns}
\begin{figure}
   \centering
   \includegraphics[width=0.7\linewidth]{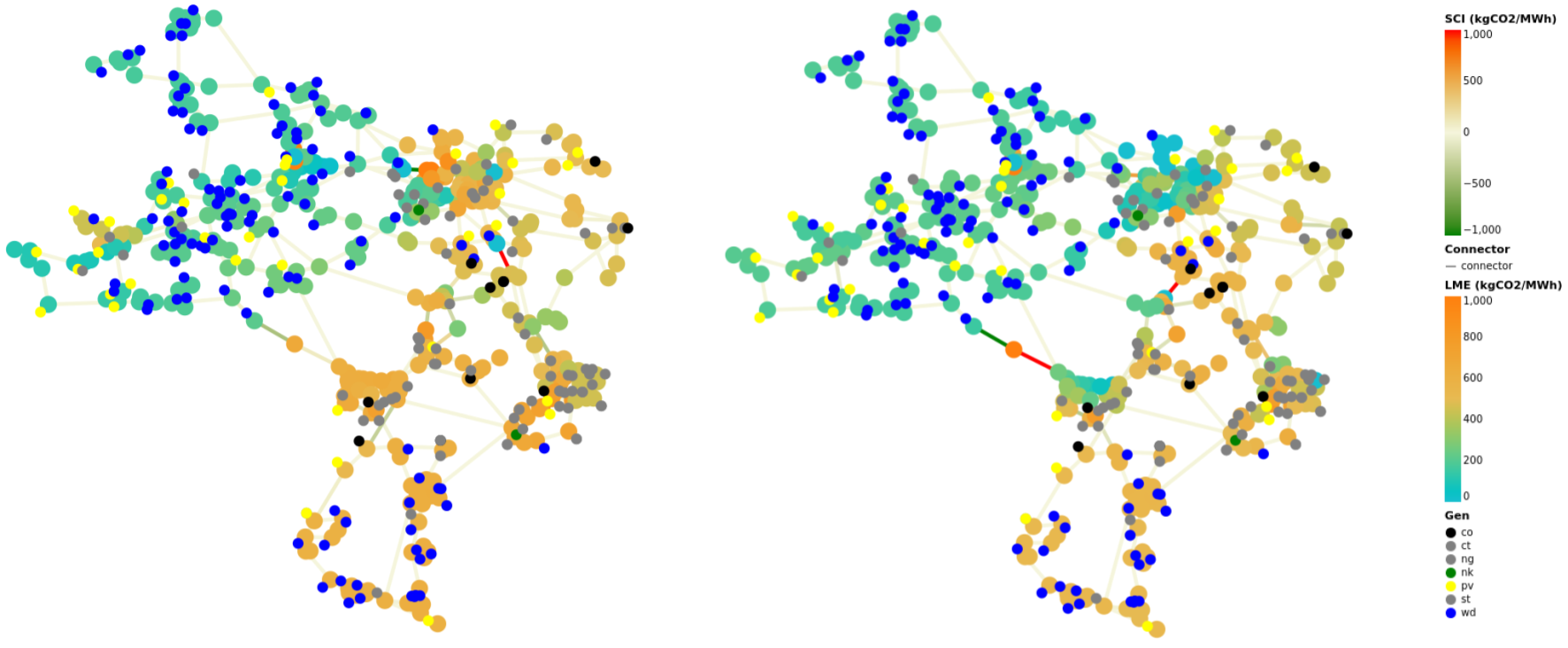}
   \caption{Average locational marginal emissions and shadow carbon intensities over the experimental horizon at 4am (left) and 3pm (right).}
\end{figure}
\paragraph{Hourly Patterns}

In general, western node LMEs tended to be relatively consistent between hours, while eastern nodes experienced lower average hourly LMEs during the afternoon and evening hours (12-20). This decrease occurs when solar production patterns are high leading to increased percentage of time with renewables on the margin. Such hourly changes in the eastern nodes LME averages indicates a significant opportunity for energy storage in these areas to reduce carbon emissions.
\begin{figure}
   \centering
   \includegraphics[width=0.7\linewidth]{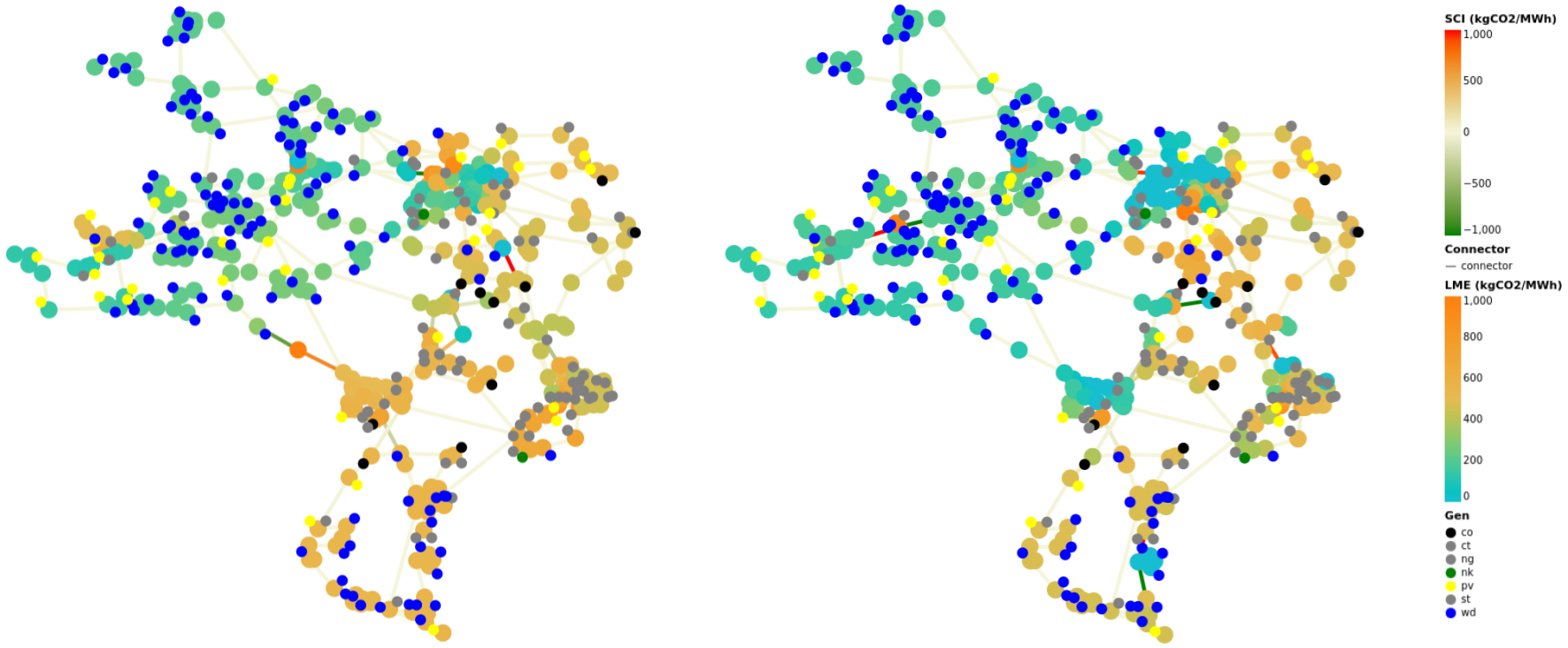}
   \caption{Average locational marginal emissions and shadow carbon intensities over the experimental horizon during January (left) and August (right).}
\end{figure}
\paragraph{Seasonality}
Expanding the temporal scale to a seasonal granularity we see similar correlations to solar generation levels with western nodes remaining largely the same across seasons and eastern nodes experiencing decreased monthly average LMEs during the summer months, when solar generation is higher. 


\subsection{ERCOT Case Study with Storage} \label{sec:storage-expts}
We can additionally apply the methods described in \Cref{sec:storage} to evaluate how LME based accounting in the ERCOT grid simulation changes with the addition of storage. Here we use 2024 storage capacity data obtained from \cite{EIA860}, matching each utility scale storage entity to the closest node and assuming a charging/discharging efficiency of 81\%. Furthermore, we split the operating horizon into single day (24 hour) increments, assuming a storage SOC of 0 at the start of each operating horizon.

As one might expect, the addition of storage allowed for an increase in total renewable production, allowing production spikes that would otherwise be transmission limited to be spread over longer periods of time, leading to an additional 20,178 GWh of wind generation and 2,962 GWh of solar generation. Interestingly, this increased consistency of renewable availability also corresponded with increased transmission constraint limitations on coal generation, with this additional 23,140 GWh of renewable production corresponding to a 83,838 GWh decrease in coal production, which in turn was offset by a 70,998 GWh increase in higher-cost natural gas generation\footnote{Note that these changes in generation do not zero out due to the net increase in power consumption caused by storage inefficiency.}. Overall, this lead to an over 20\% decrease in total system emissions, totaling 137 MMT of carbon over the experimental horizon.
\begin{table}[H]
    \centering
    \begin{tabular}{ |c||c|c|c|c|c| } 
     \hline
      Generation Type & Coal & Gas & Nuclear & Wind & Solar \\ 
        \hline
      Carbon Account (MMT CO2) & 5.5 & -5.2 & -13.0 & -24.0 & -5.8 \\ 
      Dispatch (GWh) & 31,655 & 243,941 & 39,452 & 120,783 & 17,514 \\ 
      Scope 1 Emissions (MMT CO2) & 27.38 & 109.8& 0 & 0 & 0 \\
     \hline
    \end{tabular}
    \caption{Generation resource carbon accounts and total power production with the additon of storage over the experimental horizon.}
    \label{tab:gen_accounts_stor}
\end{table}
With regards to yearlong carbon accounting, the accounts of storage infrastructure received an overall slightly net positive carbon account. This is likely due to the generally consistent nature of daily LMEs combined with a roundtrip efficiency of 81\% causing the positive carbon accounts from charging to not be fully offset by the negative carbon accounts from discharging. Additionally, the carbon accounts reflect how the introduction of storage changed the nature of transmission constraints to be coal-generation limiting, decreasing magnitude and quantity of transmission lines with negative carbon accounts, leading to an overall much higher carbon account for transmission in the storage case.
\begin{figure}[H]
    \centering
    \includegraphics[width=0.5\linewidth]{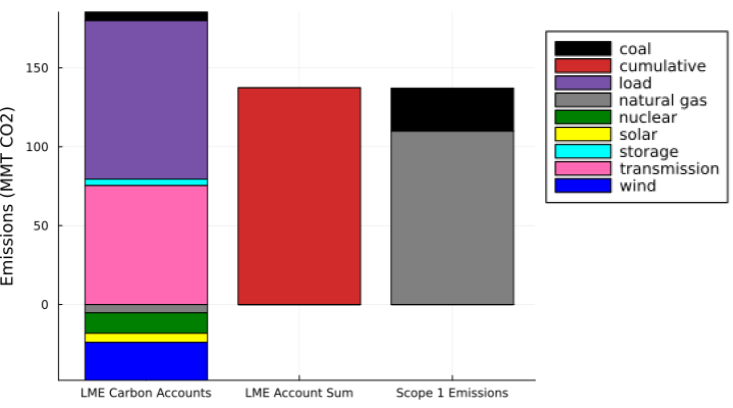}
    \caption{Total carbon accounts with the addition of storage over the experimental horizon.}
    \label{fig:accounts_stor}
\end{figure}

Despite a significant difference in transmission's role and carbon accounts, the overall regional characteristics remain largely the same with the addition of storage with western nodes continuing to display much lower LMEs than their eastern counterparts. Similarly, longer term trends such as seasonal differences remain with higher solar generation months leading to lower overall LMEs.
\begin{figure}[H]
   \centering
   \includegraphics[width=0.7\linewidth]{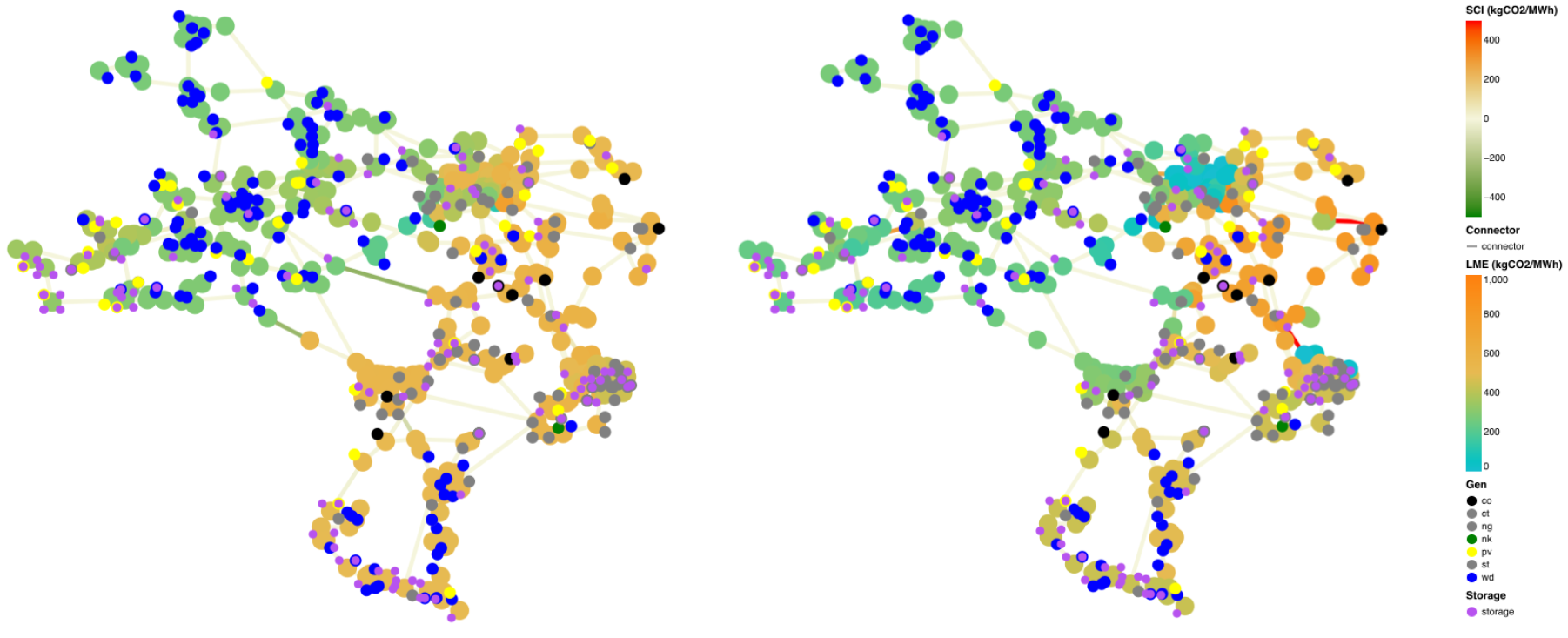}
   \caption{Average locational marginal emissions and shadow carbon intensities with the addition of storage over the experimental horizon during January (left) and August (right).}
\end{figure}
However, when considering smaller time scale trends such as hourly LMEs we see that storage operating on a daily time scale horizon acts to smooth the overall swing of LMEs, leading to more consistency across hourly periods.
\begin{figure}[H]
   \centering
    \includegraphics[width=0.7\linewidth]{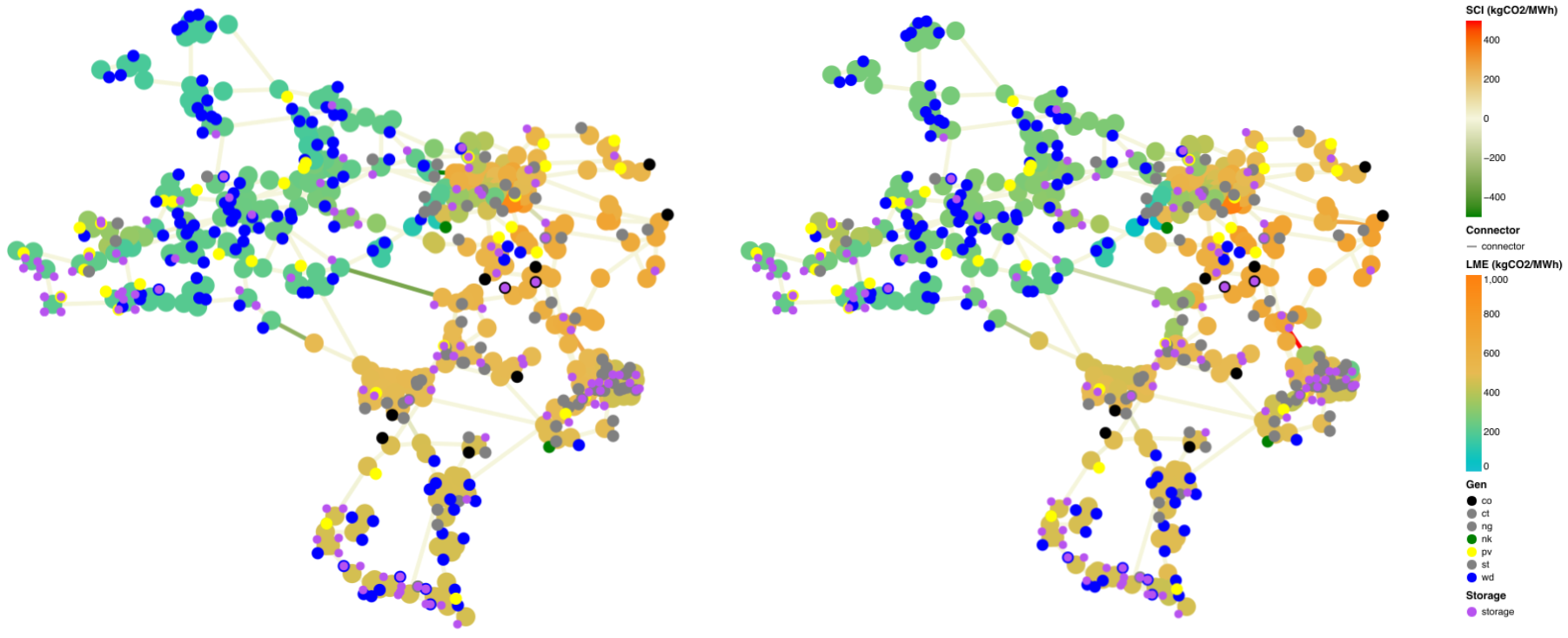}
   \caption{Average locational marginal emissions and shadow carbon intensities with the addition of storage over the experimental horizon at 4am (left) and 3pm (right).}
\end{figure}

\section{Conclusion}
\subsection{Evaluating LME Based Accounting}
Having developed intuition and key properties of LMEs along with an empirical understanding of their performance, we can now evaluate how LME based accounting has the potential to make significant improvements in the key challenge areas of current grid-level carbon accounting outlined in \Cref{sec:issues}.
\paragraph{Deliverability:} LMEs are driven by the physical power flow dispatch and its congestion constraints. Thus, LMEs are explicitly tied in definition to deliverability, reflecting the real impact of injection changes on grid power delivery. Such tendencies to reflect physical deliverability was observed in \Cref{fig:avgLME} where the LMEs clearly displayed the geographic separation, and subsequent transmission constraints, of renewable generation and load location.
\paragraph{Double Counting:} \Cref{thm:carbonfootprint} certifies that the total emissions across all system components carbon accounts will sum to the total system emissions, removing the need for data-intensive calculations like residual power generation. This ties together the consumption and market based approaches to emissions accounting and thus removes a key source of double-counting from the methodology.

\paragraph{Additionality:} The LME framework describes how to evaluate the impact of actions on the power grid, but does not provide a framework for evaluating whether or not consumer actions displace other possible developments. However, such an evaluation framework like the methodology described in the GHG project accounting manuscript \citep{ghgscope2}, can be used to filter which projects are eligible to generate and trade negative valued carbon accounts. In light of the carbon footprint theorem, imposing excess additionality criteria to qualify for negative carbon accounts leads to a lower bound of real system emissions on the sum of all carbon accounts. 

\paragraph{Impact Magnitude:} LMEs precisely describe the marginal impact of actions with regards to system carbon emissions. This provides a concrete and straightforward way to calculate the impact magnitude of different possible actions and proportionally rewards projects with respect to their real impact. As a simplified example, solar projects in a renewable heavy grid experiencing curtailment will have carbon accounts of 0 for times at which curtailment occurs, while a solar project in a dirty grid with high emission resources on the margin will have a highly negative carbon account. In \Cref{fig:avgLME} we directly observed how LMEs reflect the increased impact of renewable development in southern ERCOT over its western counterparts.

\subsection{Future Work}
LME based accounting schemes show significant promise as they continue to support the mechanisms governing modern grid dispatch (\Cref{thm:fft}, \Cref{thm:sft}) and help resolve major challenges facing current carbon accounting practices (\Cref{thm:carbonfootprint}, \Cref{sec:texas}). However, large-system analyses of LME properties remain understudied and studies moving beyond the ERCOT grid to other interconnects, or the entire US grid may lead to valuable insight regarding grid properties and decarbonization pathways. Furthermore, there has been little study into the long-term impacts on grid-level development regarding the consumer adoption of LME accounting policies. Understanding such impacts is crucial to evaluating the long-term effectiveness of marginal emission based accounting in addressing the key issues of current accounting practices. Future capacity expansion studies which integrate consumer responses to LMEs will be crucial to understanding to the long term development effects of such accounting policies.

%
%
%
\appendix
\section{Terms and Definitions}
\makenomenclature

\nomenclature[S]{\(N\)}{set of grid nodes}
\nomenclature[S]{\(G\)}{set of generators}
\nomenclature[S]{\(G_i\)}{set of generators at node $i$}
\nomenclature[S]{\(L\)}{set of transmission lines}
\nomenclature[V]{\(P^G_g\)}{power output of generator $g$}
\nomenclature[C]{\(P^{max}_g\)}{maximum power output of generator $g$}
\nomenclature[C]{\(P^{min}_g\)}{minimum power output of generator $g$}
\nomenclature[C]{\(P^D_i\)}{power demand at node $i$}
\nomenclature[V]{\(f_\ell\)}{power flow across line $\ell$}
\nomenclature[C]{\(f^{max}_\ell\)}{maximum power flow across line $\ell$}
\nomenclature[V]{\(B_\ell\)}{admittance of line $\ell$}
\nomenclature[V]{\(\theta_i\)}{voltage angle of node $i$}
\nomenclature[C]{\(c_g\)}{unit cost of power generation for generator $g$}
\nomenclature[C]{\(\sigma_g\)}{unit emissions of power generation for generator $g$}
\nomenclature[D]{\(\rho^+_\ell\)}{corresponds maximum forward flow constraint of line $\ell$}
\nomenclature[D]{\(\rho^-_\ell\)}{corresponds to maximum backward flow constraint of line $\ell$}
\nomenclature[D]{\(\gamma^+_g\)}{corresponds to maximum generation constraint of generator $g$}
\nomenclature[D]{\(\gamma^-_g\)}{corresponds to minimum generation constraint of generator $g$}
\nomenclature[D]{\(\pi^-_i\)}{corresponds to flow balance constraint of node $i$}
\nomenclature[D]{\(z_\ell\)}{corresponds to Kirchoff's Voltage Law constraint at line $\ell$}
\nomenclature[C]{\(\Psi\)}{grid power transfer distribution matrix with rows pertaining to transmission lines and columns pertaining to the impacts of nodal power injection on each line}
\printnomenclature

\section{Analyses of LMEs under Economic Dispatch with Storage} \label{sec:variants}
We begin by following the same steps as \Cref{sec:model} to explicitly reformulate \ref{p:dcOPFstorage} to create the following combined model formulation:
\begin{equation}
    \min \quad \sum_{t \in T} \sum_{g \in G}\sigma_{g,t} P^G_{g,t} \tag{Combined Storage Model}\label{p:storageCOPF}
\end{equation}
\begin{align}
    \text{s.t. }
     &c \cdot P^G - P^D \cdot \pi - (\rho^+ - \rho^-) \cdot f^{max} \notag \\& - \gamma^+ \cdot P^{max} - \gamma^- \cdot P^{min} - \beta^+\cdot E^{max} \leq 0, \\
    &B_{ij} (\theta_{i,t} - \theta_{j,t}) - f_{\ell,t} =  0, &\forall \ell=(i,j) \in L, \forall t \in T, \label{p:combstor:kvl}\\
    &\sum_{g \in G_i} P^G_{g,t} - \sum_{s \in S_i}(P^{S+}_{s,t}-P^{S-}_{s,t}) \notag\\
    &- \sum_{\ell = (i,j) \in L}f_{\ell,t} + \sum_{\ell = (j,i) \in L}f_{\ell,t}= P^D, & \forall i \in N,\forall t \in T, \label{p:combstor:fb}\\
    &f_{\ell,t} \leq f^{max}_\ell, & \forall \ell \in L, \forall t \in T, \\
    &f_{\ell,t} \geq -f^{max}_\ell, & \forall \ell \in L,\forall t \in T, \\
    &P^G_{g,t} \leq P^{max}_{g,t}, & \forall g \in G, \forall t \in T,\\
    &P^G_{g,t} \geq P^{min}_{g,t}, & \forall g \in G, \forall t \in T,\\
    &E_{s,t} \geq 0, & \forall s \in S, \forall t \in T,\\
    &E_{s,t} \leq E^{max}_s, & \forall s \in S, \forall t \in T,\\
    &E_{s,t}+\eta P^{S+}_{s,t} - \frac{1}{\eta}P^{S-}_{s,t}-E_{s,t+1} = 0, & \forall s \in S, \forall t \in T, \label{p:combstor:soc}\\
    &P^{S+},P^{S-} \geq 0 \\
    &\gamma^+_{g,t} + \gamma^-_{g,t} + \sum_{i \mid g \in G_i} \pi_{i,t}  = c_g, &\forall g \in G,\forall t \in T,  \label{p:combstor:gd} \\
    &z_{\ell,t} + \rho^+_{\ell,t} + \rho^-_{\ell,t} - \pi_{i,t} + \pi_{j,t} = 0, &\forall \ell=(i,j)\in L,\forall t \in T, \label{p:combstor:fd} \\
    &\sum_{\ell = (i,j) \in L} B_{\ell,t} z_{\ell,t} - \sum_{\ell = (j,i) \in L} B_{\ell,t} z_{\ell,t} = 0, &\forall i \in N,\forall t \in T,  \\
    &-\sum_{i \mid s \in S_i}\pi_{i,t} +\eta \alpha_{s,t} + \tau^+_{s,t}= 0, &\forall s \in S, \forall t \in T, \label{p:combstor:api}\\
     &\sum_{i \mid s \in S_i}\pi_{i,t} -\frac{1}{\eta} \alpha_{s,t} +\tau^-_{s,t}= 0, &\forall s \in S, \forall t \in T ,\label{p:combstor:apim}\\
    &\beta^+_{s,t}+\beta^-_{s,t}+\alpha_{s,t}-\alpha_{s,t-1}=0, &\forall s \in S, \forall t \in T,\label{p:combstor:ba}\\\
    &\rho^+, \gamma^+,\beta^+ \leq 0, \\
    &\rho^-,\gamma^-,\beta^-,\tau^+,\tau^- \geq 0.
\end{align} 
This allows us to proceed to the proofs for key LME accounting properties with storage.
\begin{proof}{of \Cref{eq:carbonfootprintstorage}}
    We begin by constructing the dual problem to the \ref{p:storageCOPF}:
    \begin{align}
        \max \quad \sum_{t \in T} ( &\sum_{\ell \in L} (p^+_{\rho,\ell,t}-p^-_{\rho,\ell,t})f^{max} + \sum_{g \in G} p^+_{\gamma,g,t} P^{max}_g + p^-_{\gamma,g,t}P^{min}_g \notag \\&
        + \sum p_{PG,g,t} c_g + \sum_{i \in N} p_{\pi,i,t}P^D_i +\sum_{s \in S}p^+_{\beta,s,t} E^{max}_{s} ) \tag{Combined Storage Dual} \label{p:CDualOPFstorage}
    \end{align}
    \begin{align}
        \text{s.t. }
        &-p_{o} c_g + p_{\pi,g,t} + p^+_{\gamma,g,t} + p^-_{\gamma,g,t} = \sigma_g, &\forall g \in G, \forall t \in T, \\
        &B_{ij} (p_{\theta,i,t} - p_{\theta,j,t}) - p_{f,\ell,t} =  0, &\forall \ell=(i,j) \in L, \forall t \in T,\\
        &\sum_{g \in G_i} p_{PG,g,t} - \sum_{s \in S_i} p^+_{PS,s,t}+p^-_{PS,s,t} \notag \\
        &- \sum_{\ell = (i,j) \in L}p_{f,\ell,t} + \sum_{\ell = (j,i) \in L}p_{f,\ell,t}= p_o P^D_i, & \forall i \in N, \forall t \in T,\\
        &p_{f,\ell,t} \leq p_of^{max}_\ell, &\forall \ell \in L, \forall t \in T,\\
        &p_{f,\ell,t} \geq -p_of^{max}_\ell, &\forall \ell \in L, \forall t \in T, \\
        &p_{PG,g,t} \leq p_oP^{max}_g, &\forall g \in G, \forall t \in T,\\
        &p_{PG,g,t} \geq p_oP^{min}_g, &\forall g \in G, \forall t \in T,\\
        &p_{E,s,t} \geq 0, &\forall s \in S, \forall t \in T,\\
        &p_{E,s,t} \leq p_o E^{max}_s, &\forall s \in S, \forall t \in T,\\
        &\eta p^+_{PS,s,t}-\frac{1}{\eta}p^-_{PS,s,t}+p_{E,s,t} = p_{E,s,t+1}, &\forall s \in S, \forall t \in T,\\
        &p^+_{PS}, p^-_{PS} \geq 0, &\forall s \in S, \forall t \in T,\\
        &p_{z,\ell,t} + p^+_{\rho,\ell,t} + p^-_{\rho,\ell,t} - p_{\pi,i,t} + p_{\pi,j,t} = 0, &\forall \ell=(i,j) \in L, \forall t \in T, \label{p:dualstor:fd}\\
        &\sum_{\ell = (i,j) \in L} B_\ell p_{z,\ell,t} - \sum_{\ell = (j,i) \in L} B_\ell p_{z,\ell,t} = 0, &\forall i \in N, \forall t \in T, \label{p:dualstor:td}\\
        &\eta p_{\alpha,s,t}-\sum_{i \mid s \in S_i}p_{\pi,i,t} +p^+_{\tau,s,t} = 0 &\forall s \in S, \forall t \in T\\
        &-\frac{1}{\eta}p_{\alpha,s,t}+\sum_{i \mid s \in S_i}p_{\pi,i,t}+ p^-_{\tau,s,t} = 0, &\forall s \in S, \forall t \in T,\\
        &p^+_{\beta,s,t}+p^-_{\beta,s,t}+p_{\alpha,s,t}-p_{\alpha,s,t-1}=0, &\forall s \in S, \forall t \in T,\\
        &p^+_{\rho}, p^+_{\gamma}, p^+_{\beta} \leq 0, \\
        &p^-_{\rho}, p^-_{\gamma}, p^+_{\tau}, p^-_{\tau},  p^-_{\beta} \geq 0.
\end{align}
    We now proceed to evaluate the sum of carbon accounts:
    \begin{equation}
        \sum_{t\in T}\left(\sum_{i \in N}LME_{i,t} P^D_{i,t} + \sum_{g \in G}(\sigma_{g,t} - LME_{g,t}) P^G_{g,t} + \sum_{\ell \in L} SCI_{\ell,t} |f_{\ell,t}|+\sum_{s \in S} LME_{s,t} (P^{S+}_{s,t}- P^{S-}_{s,t}) \right)
    \end{equation}
    where plugging in definition of $LME_{i,t}$ and $SCI_{i,t}$ gives
    \begin{align}
        \sum_{t\in T}\Bigg( &\sum_{i \in N}\left((p_{\pi,i,t}+p_o\pi_{i,t}) P^D_{i,t}\right) + \sum_{g \in G}\left((\sigma_g - p_{\pi,g,t}-p_o\pi_{g,t}) P^G_{g,t}\right) \notag \\
        & + \sum_{\ell \in L} \left( (p^+_{\rho,\ell,t}+p^-_{\rho,\ell,t}+p_o(\rho_{\ell,t}^++\rho_{\ell,t}^-))f_{\ell,t}\right)  +\sum_{s \in S} \left((p_{\pi,s,t}+p_o\pi_{s,t}) (P^{S+}_{s,t}- P^{S-}_{s,t}) \right)\Bigg).
    \end{align}
    Substituting $\pi_{s,t}$ using \Cref{p:combstor:api} and \Cref{p:combstor:apim} along with complementary slackness of $\tau$.
    \begin{align}
        \sum_{t\in T}\Bigg( &\sum_{i \in N}\left((p_{\pi,i,t}+p_o\pi_{i,t}) P^D_{i,t}\right) + \sum_{g \in G}\left((\sigma_g - p_{\pi,g,t}-p_o\pi_{g,t}) P^G_{g,t}\right)\notag \\&+ \sum_{\ell \in L} \left( (p^+_{\rho,\ell,t}+p^-_{\rho,\ell,t}+p_o(\rho_{\ell,t}^++\rho_{\ell,t}^-))f_{\ell,t}\right) \notag \\
        & +\sum_{s \in S} \left((p_{\pi,s,t}+p_o\eta\alpha_{s,t}) P^{S+}_{s,t}- (p_{\pi,s,t}+p_o\frac{1}{\eta}\alpha_{s,t})P^{S-}_{s,t}) \right)\Bigg).
    \end{align}
    Applying \Cref{p:combstor:ba} to substitute $\alpha_{s,t}$
    \begin{align}
        \sum_{t\in T}\Bigg( &\sum_{i \in N}\left((p_{\pi,i,t}+p_o\pi_{i,t}) P^D_{i,t}\right) + \sum_{g \in G}\left((\sigma_g - p_{\pi,g,t}-p_o\pi_{g,t}) P^G_{g,t}\right)\notag \\&+ \sum_{\ell \in L} \left( (p^+_{\rho,\ell,t}+p^-_{\rho,\ell,t}+p_o(\rho_{\ell,t}^++\rho_{\ell,t}^-))f_{\ell,t}\right) \notag \\
        & +\sum_{s \in S} \Big((p_{\pi,s,t}+p_o\eta(\beta^+_{s,t+1}+\beta^-_{s,t+1}+\alpha_{s,t+1})) P^{S+}_{s,t}\notag \\&- (p_{\pi,s,t}+p_o\frac{1}{\eta}(\beta^+_{s,t+1}+\beta^-_{s,t+1}+\alpha_{s,t+1}))P^{S-}_{s,t}) \Big)\Bigg)
    \end{align}
    and from subsequent recursive expansion of $\alpha_{s,t}$ using \Cref{p:combstor:ba}
    \begin{align}
        \sum_{t\in T}\Bigg(&\sum_{i \in N}\left((p_{\pi,i,t}+p_o\pi_{i,t}) P^D_{i,t}\right) + \sum_{g \in G}\left((\sigma_g - p_{\pi,g,t}-p_o\pi_{g,t}) P^G_{g,t}\right)\notag \\& + \sum_{\ell \in L} \left( (p^+_{\rho,\ell,t}+p^-_{\rho,\ell,t}+p_o(\rho_{\ell,t}^++\rho_{\ell,t}^-))f_{\ell,t}\right) \notag \\ 
        &+\sum_{s \in S} \left((p_{\pi,s,t}) (P^{S+}_{s,t}- P^{S-}_{s,t})+p_o(\beta^+_{s,t}+\beta^-_{s,t}) \sum_{t'=1}^{t-1}(\eta P^{S+}_{s,t'}- \frac{1}{\eta}P^{S-}_{s,t'}) \right) \Bigg),
    \end{align}
    using \Cref{p:combstor:soc} to substitute for $\sum_{t'=1}^{t-1}(\eta P^{S+}_{s,t'}- \frac{1}{\eta}P^{S-}_{s,t'})$
    \begin{align}
        \sum_{t\in T}\Bigg(&\sum_{i \in N}\left((p_{\pi,i,t}+p_o\pi_{i,t}) P^D_{i,t}\right) + \sum_{g \in G}\left((\sigma_g - p_{\pi,g,t}-p_o\pi_{g,t}) P^G_{g,t}\right)\notag\\& + \sum_{\ell \in L} \left( (p^+_{\rho,\ell,t}+p^-_{\rho,\ell,t}+p_o(\rho_{\ell,t}^++\rho_{\ell,t}^-))f_{\ell,t}\right) \notag \\
        &+\sum_{s \in S} \left((p_{\pi,s,t}) (P^{S+}_{s,t}- P^{S-}_{s,t})+p_o(\beta^+_{s,t}+\beta^-_{s,t}) E_{s,t} \right) \Bigg).
    \end{align}
    Now, using \Cref{p:combstor:gd} and strong duality of the storage DC-OPF dispatch problem this simplifies to
    \begin{align}
        \sum_{t\in T}\Bigg(&\sum_{i \in N}\left((p_{\pi,i,t}) P^D_{i,t}\right) + \sum_{g \in G}\left((\sigma_g - p_{\pi,g,t}) P^G_{g,t}\right) + \sum_{\ell \in L} \left( (p^+_{\rho,\ell,t}+p^-_{\rho,\ell,t})f_{\ell,t}\right)\notag \\& +\sum_{s \in S} \left((p_{\pi,s,t}) (P^{S+}_{s,t}- P^{S-}_{s,t}) \right)\Bigg).
    \end{align}
    Substituting $(p^+_{\rho,\ell,t}+p^-_{\rho,\ell,t})$ using \Cref{p:dualstor:fd} leads to
    \begin{align}
        \sum_{t\in T}\Bigg(&\sum_{i \in N}\left((p_{\pi,i,t}) P^D_{i,t}\right) + \sum_{g \in G}\left((\sigma_g - p_{\pi,g,t}) P^G_{g,t}\right) + \sum_{\ell \in L} \left( (p_{\pi,i,t}-p_{\pi,j,t}-p_{z,\ell,t})f_{\ell,t}\right) \notag\\& +\sum_{s \in S} \left((p_{\pi,s,t}) (P^{S+}_{s,t}- P^{S-}_{s,t}) \right)\Bigg),
    \end{align}
    rearranging to
    \begin{align}
        \sum_{t\in T}\Bigg(& \sum_{i \in N}\left(p_{\pi,i,t} (P^D_{i,t}-\sum_{g \in G_i}P^G_{g,t}+\sum_{s \in S_i}(P^{S+}_{s,t}- P^{S-}_{s,t})+\sum_{\ell=(i,j)\in L}f_{\ell,t}-\sum_{\ell=(j,i)\in L}f_{\ell,t})\right) + \sum_{g \in G}\left(\sigma_g P^G_{g,t}\right) \notag\\
        &- \sum_{\ell=(i,j) \in L} \left(p_{z,\ell,t}f_{\ell,t} \right)\Bigg).
    \end{align}
    Where plugging in \Cref{p:combstor:fb} simplifies to
    \begin{equation}
        \sum_{t\in T}\left(\sum_{g \in G}\left(\sigma_g P^G_{g,t}\right) - \sum_{\ell=(i,j) \in L} \left(p_{z,\ell,t}f_{\ell,t} \right)\right).
    \end{equation}
    Finally, plugging in \Cref{p:combstor:kvl} rearranging the right sum gives
    \begin{equation}
        \sum_{t\in T}\left(\sum_{g \in G}\left(\sigma_g P^G_{g,t}\right) - \sum_{i\in N} \theta_{i,t} \left(\sum_{\ell=(i,j)\in L}p_{z,\ell,t} B_{\ell}-\sum_{\ell=(j,i)\in L}p_{z,\ell,t}B_{\ell} \right)\right)
    \end{equation}
    which \Cref{p:dualstor:td} simplifies to
    \begin{equation}
        \sum_{t\in T}\left(\sum_{g \in G}\sigma_g P^G_{g,t}\right).
    \end{equation}
     
\end{proof}
\begin{corollary}
    Note that $(P^{S+}_{s,t}-P^{S-}_{s,t})$ is precisely the net injection that storage element $s$ provides to the grid at time $t$. This can be equivalently viewed as either a load or carbon free generation resource (depending on if the injection is positive or negative), and solved at the single period dispatch level with an equivalent carbon account for the storage resource. Such an equivalence implies that not only do the sum of carbon accounts over the entire multiperiod dispatch sum to the total carbon emissions, but also the sum of carbon accounts at each dispatch period sum to the total carbon emissions which occur at that time period.
\end{corollary}

\begin{proof}{of \Cref{thm:stor-fft}}
We proceed by contradiction. Let $\hat{X} = (\hat{P}^G, \hat{P}^{S+}, \hat{P}^{S-}, \hat{f}, \hat{\theta})$ denote the arbitrary optimal primal solution to \ref{p:storageCOPF} used to derive the pricing parameters $\hat{LME}$ and $\hat{LMP}$. Let $X = (P^G, P^{S+}, P^{S-}, f, \theta)$ be a candidate solution satisfying the conditions of Theorem \ref{thm:stor-fft}.

Assume for the sake of contradiction that $X$ is not an optimal solution to \ref{p:storageCOPF}. Since $X$ satisfies the primary economic optimality conditions (implied by the First Fundamental Theorem of Welfare Economics given $\hat{LMP}$), this non-optimality implies strictly higher system emissions:
\begin{equation}
    \sum_{g \in G} \sigma_g P_g^G > \sum_{g \in G} \sigma_g \hat{P}_g^G \implies \sum_{g \in G} \sigma_g \Delta P_g^G > 0.
\end{equation}

First, we establish the optimality of the decentralized storage response. From \Cref{thm:stor-sft}, the  solution $(\hat{P}^{S+}, \hat{P}^{S-})$ is necessarily a minimizer of the decentralized storage problem \eqref{eq:stor-dec}. By hypothesis, the candidate solution $(P^{S+}, P^{S-})$ is also an optimal solution to \eqref{eq:stor-dec}. Therefore, they must share the same optimal objective value:
\begin{equation} \label{eq:stor-cost-equality}
    \sum_{t \in T} \sum_{s \in S} \hat{LME}_{s,t} (P_{s,t}^{Net}) = \sum_{t \in T} \sum_{s \in S} \hat{LME}_{s,t} (\hat{P}_{s,t}^{Net}) \implies \sum_{t \in T} \sum_{s \in S} \hat{LME}_{s,t} \Delta P_{s,t}^{Net} = 0
\end{equation}
where $P^{Net} = P^{S+} - P^{S-}$.

Now, consider the decentralized generator response. Following the logic of \Cref{thm:fft}, for any generator maximizing its carbon profit given $\hat{LME}$, we have the inequality:
\begin{equation} \label{eq:gen-ineq}
    \sum_{g \in G} \sigma_g \Delta P_g^G \leq \sum_{g \in G} \hat{LME}_g \Delta P_g^G
\end{equation}

We next examine network feasibility. The dual variables $\hat{LME}$ are constructed from the Lagrange multipliers of the flow balance and transmission constraints. For any feasible change in injections satisfying \eqref{eq:stor:voltage}-\eqref{eq:stor:fb}, the value of the net injection vector priced at $\hat{LME}$ must be non-positive.
\begin{equation}
    \sum_{i \in N} \hat{LME}_i \Delta Inj_i \le 0
\end{equation}
Substituting the components of the net injection $\Delta Inj = \Delta P^G - \Delta P^{Net}$:
\begin{equation} \label{eq:net-feas}
    \sum_{g \in G} \hat{LME}_g \Delta P_g^G - \sum_{s \in S} \hat{LME}_s \Delta P_s^{Net} \le 0
\end{equation}
Substituting the storage optimality result \eqref{eq:stor-cost-equality} into \eqref{eq:net-feas} yields:
\begin{equation}
    \sum_{g \in G} \hat{LME}_g \Delta P_g^G \le 0
\end{equation}
Finally, combining this with the generator inequality \eqref{eq:gen-ineq}:
\begin{equation}
    \sum_{g \in G} \sigma_g \Delta P_g^G \le \sum_{g \in G} \hat{LME}_g \Delta P_g^G \le 0
\end{equation}
This implies $\sum \sigma_g \Delta P_g^G \le 0$, which contradicts the assumption that $X$ yields strictly higher emissions. Thus, $X$ must be an optimal solution to \ref{p:storageCOPF}.
\end{proof}

\begin{proof}{of \Cref{thm:stor-sft}}

Let us define $\hat{P}^G,\hat{P}^{S-},\hat{P}^{S-},\hat{f},\hat{\theta}, \hat{\pi}, \hat{\gamma}^+,\hat{\gamma}^-,\hat{\rho}^+,\hat{\rho}^-,\hat{\alpha},\hat{\beta}^+,\hat{\beta}^-$, $\hat{p}$ to be the optimal primal and dual solutions to the \ref{p:storageCOPF} which correspond to the $\hat{LME}$ and $\hat{LMP}$ values.

We begin by noting that the economic account of an arbitrary storage operator, $s\in S$ associated with the solution, $(\hat{P}^{S-}_s,\hat{P}^{S+}_s)$ will reside in $P^{S*}_s$, as the economic dispatch problem with storage is known to support a decentralized equilibrium \cite{jiang2023duality}.

We now examine the associated carbon account of the storage operator
\begin{equation}
    \sum_{t \in T}\hat{LME_{s,t}} \cdot (\hat{P}^{S+}_{s,t}-\hat{P}^{S-}_{s,t}).
\end{equation}
Where expanding the definition of LME simplifies to
\begin{equation}
    \sum_{t \in T}(\hat{p}_o \hat{\pi}_{s,t}+\hat{p}_{\pi,s,t}) \cdot (\hat{P}^{S+}_{s,t}-\hat{P}^{S-}_{s,t}).
\end{equation}
Using the same recursive substitution as in the proof of \Cref{eq:carbonfootprintstorage} gives
\begin{equation}
 \sum_{t \in T}(\hat{p}_o (\sum_{t'=t+1}^T( \hat{\beta}^+_{s,t'}+\hat{\beta}^{-}_{s,t'}))+\sum_{t'=t+1}^T(\hat{p}^+_{\beta,s,t'}+\hat{p}^-_{\beta,s,t'})) \cdot (\eta \hat{P}^{S+}_{s,t}-\frac{1}{\eta}\hat{P}^{S-}_{s,t}).
\end{equation}
Which the SOC definition simplifies to
\begin{equation}
    \sum_{t \in T}(\hat{p}_o (\hat{\beta}^+_{s,t}+\hat{\beta}^{-}_{s,t})+(\hat{p}^+_{\beta,s,t}+\hat{p}^-_{\beta,s,t})) \cdot \hat{E}_{s,t}.
\end{equation}
And finally, from complementary slackness of $\hat{\beta}^+$, $\hat{\beta}^-$, $\hat{p}_{\beta}^+$, and $\hat{p}_{\beta}^-$ we get
\begin{equation}
    \sum_{t \in T}(\hat{p}_o (\hat{\beta}^+_{s,t})+(\hat{p}^+_{\beta,s,t})) \cdot E^{max}_{s,t}.
\end{equation}
Now, to show that this is indeed a minimizing solution to \ref{eq:stor-dec} we can consider any arbitrary solution to \ref{eq:stor-dec} defined by $(P^{S+}_s, P^{S-}_s) \in P^{S*}_s$:
\begin{align}
    \sum_{t \in T}\hat{LME}_{s,t} \cdot (P^{S+}_{s,t}-P^{S-}_{s,t}) &= \sum_{t \in T}(\hat{p}_o \hat{\pi}_{s,t}+\hat{p}_{\pi,s,t}) \cdot (P^{S+}_{s,t}-P^{S-}_{s,t}). 
\end{align}
Noting that $\frac{1}{\eta}\hat{\alpha}_{s,t} \geq \hat{\pi}_{s,t}\geq \eta \hat{\alpha}_{s,t}$ and $\frac{1}{\eta}\hat{p}_{\alpha,s,t} \geq \hat{p}_{\pi,s,t}\geq \eta \hat{p}_{\alpha,s,t}$ leads to
\begin{align}
    \sum_{t \in T}\hat{LME}_{s,t} \cdot (P^{S+}_{s,t}-P^{S-}_{s,t}) &\geq \sum_{t \in T}(\hat{p}_o \hat{\alpha}_{s,t}+\hat{p}_{\alpha,s,t}) \cdot (\eta P^{S+}_{s,t}-\frac{1}{\eta}P^{S-}_{s,t}), \\ 
    &= \sum_{t \in T}(\hat{p}_o (\sum_{t'=t}^T( \hat{\beta}^+_{s,t'}+\hat{\beta}^{-}_{s,t'}))+\sum_{t'=t}^T(\hat{p}^+_{\beta,s,t'}+\hat{p}^-_{\beta,s,t'})) \cdot (\eta P^{S+}_{s,t}-\frac{1}{\eta}P^{S-}_{s,t}), \\
    &= \sum_{t \in T}(\hat{p}_o (\hat{\beta}^+_{s,t}+\hat{\beta}^{-}_{s,t})+(\hat{p}^+_{\beta,s,t}+\hat{p}^-_{\beta,s,t})) \cdot E_{s,t}, \\
    &\geq \sum_{t \in T}(\hat{p}_o (\hat{\beta}^+_{s,t})+(\hat{p}^+_{\beta,s,t})) \cdot E^{max}_{s,t}.
\end{align}
Which, as shown above, is precisely the value of the storage account for $s$ associated with ($\hat{P}^{S-}_s,\hat{P}^{S+}_s$). Thus, the storage account for $s$ associated with ($\hat{P}^{S-}_s,\hat{P}^{S+}_s$) must indeed be a minimizing solution to \ref{eq:stor-dec}.

\end{proof}

\section{Sensitivity Analysis Lemmas}\label{sec:sensan}
In this section we present a key sensitivity analysis lemma in the style of the envelope theorem which we rely upon for the analysis of the \ref{p:OPF} and the subsequent definitions of LMEs and SCIs. Consider a convex optimization problem described in the form
    \begin{align*}
        \min \quad & c(x)\\
        \text{s.t. } \quad 
        &Ax = b, \\
        &x \geq 0,
    \end{align*}
with a proper, differentiable, closed convex function $c:\R^n\to\R$ and a matrix $A\in\R^{m\times n}$. Define $P(\alpha)$ to be the perturbed problem created by replacing $A$ by $(A+\alpha G)$ for an arbitrary matrix $G\in \R^{m\times n}$ and $\alpha\in\R$. That is,
    \begin{align*}
  P(\alpha):\quad z(\alpha)=\min \quad & c(x)\\
                  \text{s.t. } \quad  & (A+\alpha G) x = b, \\
                                      & x \geq 0.
    \end{align*}
Define the Lagrangian dual problem of $P(\alpha)$ as
\begin{align}
    \sup_{p\in\R^m, \lambda\in\R^n_+}\;\inf_{x\in\R^n} \; c(x) - p^\top((A+\alpha G)x-b) - \lambda^\top x.
\end{align}
\begin{lemma}\label{thm:envelope}
     Then for any set of primal optimal solutions $x$ and dual optimal solutions $p, \lambda$ that satisfy the following assumptions:
    \begin{enumerate}
        \item the directional derivative, $\frac{d z}{d \alpha^+}$ exists and is finite valued;
        \item there exists a convergent sequence $\{\alpha^{(i)}\}$ satisfying $\alpha^{(i)} > \alpha, \lim_{i\rightarrow \infty}\alpha^{(i)} = \alpha$ with a corresponding sequence of primal and dual optimal solutions, $x^{(i)}, p^{(i)}, \lambda^{(i)}$ such that $\lim_{i \rightarrow \infty}x^{(i)} = x$, $\lim_{i \rightarrow \infty}p^{(i)} = p$, and $\lim_{i \rightarrow \infty}\lambda^{(i)} = \lambda$,
    \end{enumerate}
   the following derivative expression holds
    \begin{equation}
        z'(\alpha)^+=-p^T G x.
    \end{equation}
\end{lemma}

\begin{proof}
    From the first-order optimality conditions (KKT conditions) of the Lagrangian at the optimal solution $x$ for parameter $\alpha$, we have:
    \begin{align}
        \nabla_x L(x, p, \lambda, \alpha) &= 0 \\
        \implies \nabla c(x) &= (A+\alpha G)^T p + \lambda. \label{eq:kkt_grad}
    \end{align}
    We now evaluate the directional derivative of the value function using the definition of the limit. By Taylor expansion of the objective function $c(x)$, we can approximate the difference in the objective values:
    \begin{align}
        z'(\alpha)^+ &= \lim_{i \rightarrow \infty} \frac{z(\alpha^{(i)})-z(\alpha)}{\alpha_i-\alpha} \\
        &= \lim_{i \rightarrow \infty} \frac{c(x^{(i)})-c(x)}{\alpha^{(i)}-\alpha} \\
        &= \lim_{i \rightarrow \infty} \frac{\nabla c(x)^T (x^{(i)}-x) + o(\|x^{(i)}-x\|)}{\alpha_i-\alpha},
    \end{align}
    where $\lim_{i\rightarrow \infty}\frac{o(||x^{(i)}-x||)}{\alpha^{(i)}-\alpha}=0$ from the finite valued property of $\frac{dz}{d\alpha^+}$. Substituting the expression for $\nabla c(x)$ from Eq. \eqref{eq:kkt_grad}:
    \begin{align}
        z'(\alpha)^+ &= \lim_{i \rightarrow \infty} \frac{\left((A+\alpha G)^T p + \lambda\right)^T (x^{(i)}-x)}{\alpha^{(i)}-\alpha} \\
        &= \lim_{i \rightarrow \infty} \frac{(p)^T (A+\alpha G)(x^{(i)}-x) + (\lambda)^T (x^{(i)}-x)}{\alpha^{(i)}-\alpha}. \label{eq:deriv_split}
    \end{align}
    We analyze the feasibility term $(A+\alpha G)(x^{(i)}-x)$ separately. Since both $x^{(i)}$ and $x$ are feasible for their respective problems, we have:
    \begin{align}
        (A+\alpha^{(i)} G)x^{(i)} &= b \\
        (A+\alpha G)x &= b.
    \end{align}
    Rewriting the first equation by splitting $\alpha_i$:
    \begin{align}
        (A+\alpha G + (\alpha^{(i)} - \alpha)G)x^{(i)} &= b \\
        (A+\alpha G)x^{(i)} + (\alpha_i - \alpha)G x^{(i)} &= b.
    \end{align}
    Subtracting the feasibility condition for $x^*$ from this equation:
    \begin{align}
        (A+\alpha G)(x^{(i)}-x) + (\alpha_i - \alpha)G x^{(i)} &= 0 \\
        \frac{(A+\alpha G)(x^{(i)}-x)}{\alpha_i-\alpha} &= -G x^{(i)}.
    \end{align}
    Taking the limit as $n \to \infty$ (where $x^{(i)} \to x$):
    \begin{equation}
        \lim_{i \rightarrow \infty} \frac{(A+\alpha G)(x^{(i)}-x)}{\alpha^{(i)}-\alpha} = -G x. \label{eq:feas_limit}
    \end{equation}
    Now, returning to the term involving $\lambda$: by complementary slackness, $\lambda^T x = 0$. Since $\lambda \geq 0$ and $x^{(i)} \geq 0$, we have $\lambda^T x^{(i)} \geq 0$. Thus, $\lambda^T (x^{(i)} - x) = \lambda^T x^{(i)} \geq 0$. With such a non-negativity property we can now expand the expression to 
    \begin{equation}
         \lim_{i \rightarrow \infty} \sum_{k\in [n]}\frac{\lambda _k x^{(i)}_k}{\alpha^{(i)}-\alpha} = \sum_{k\in [n]}\lim_{i \rightarrow \infty}\frac{\lambda_k x^{(i)}_k}{\alpha^{(i)}-\alpha}.
    \end{equation}
    Now, fix an arbitrary $k \in [n]$ and consider $\lambda_k x^{(i)}_k$, i.e. the $kth$ element of $x^{(i)}$ multiplied by the $kth$ element of $\lambda$. Note that if $\lambda_k>0$, $\lim_{i \rightarrow \infty}\lambda_i=\lambda$ implies that there exists some $C>0$ such that $\forall i \geq C, \lambda^{(i)}_k>0$ which from complementary slackness implies $x^{(i)}_k=0$. Additionally, from its definition $\forall n, \alpha^{(i)}>\alpha$ allowing us to complete our evaluation of the $\lambda$ term as 
    \begin{equation} \label{eq:lmbdterm}
        \sum_{k\in [n]}\lim_{i \rightarrow \infty}\frac{\lambda_k x^{(i)}_k}{\alpha^{(i)}-\alpha} = 0.
    \end{equation}
    Plugging \eqref{eq:feas_limit} and \eqref{eq:lmbdterm} back into \eqref{eq:deriv_split} yields:
    \begin{align}
        z'(\alpha)^+ &= p^T (-G x) + 0 \\
        &= -p^T G x.
    \end{align}
\end{proof}
When applied to a linear program where we consider perturbations in the coefficient matrix, \Cref{thm:envelope} specifies to \Cref{thm:freund} which closely echoes that presented in \cite{freund2009postoptimal}. As the \ref{p:OPF} and the \ref{p:storageCOPF} presented in this paper are both linear programs \Cref{thm:freund} works to describe LMEs in this manner, however, \ref{thm:envelope} enables the generalization of such methods to nonlinear convex cost functions.
\section{Properties of the Combined Method}\label{sec:combinedProp}
In this section we study the properties of a more general form of the combined model, designed to reflect only the two-tiered optimization model structure. Without loss of generality we assume the standard DC OPF problem to be expressed in the form:
\begin{align}
    \min \quad &c^T x \tag{Primal LP} \label{p:PrLP}\\
    \text{s.t. }\notag
    &Ax = b, &[p],\\
    &x \geq 0,
\end{align} 
where $x\in \mathbb{R}^n$ is a vector of our decision variables, $A\in\mathbb{R}^{m\times n}$ is a matrix, and $c\in \mathbb{R}^n$ and $b\in \mathbb{R}^m$ are vectors of constants. This yields a combined problem of the following form: 
\begin{align}
    \min \quad &\hat{c}^T x \tag{Combined LP} \label{p:CLP}\\
    \text{s.t. }
    &Ax = b, & [\pi_P],\label{eq:p} \\
    &x \geq 0, \label{eq:nn}\\
    &A^Tp \leq c, &[\pi_D],\label{eq:d}\\
    &cx-b^Tp = 0, &[\pi_O],\label{eq:o}
\end{align}
where $\hat{c}$ reflects the carbon intensities of the generators. 

\begin{lemma}\label{s:CombinedEQ}
    Under assumptions of uniqueness and nondegeneracy, the LME definition of \Cref{thm:LME} is equivalent to the LME definition of \cite{ruiz2010analysis}
\end{lemma}
\begin{proof}
We begin by describing the Lagrangian dual problem to the \ref{p:CLP}: 
\begin{align}
    \text{maximize} \quad &c^T \pi_D + b^T \pi_P \tag{Combined LP Dual} \label{p:DCLP}\\
    \text{s.t. }
    &A\pi_D - b\pi_O=0, \\
    &\pi_D \geq 0, \\
    &A^T\pi_P+c\pi_O \leq \hat{c},
\end{align}
where $\pi_O$ corresponds to \eqref{eq:o}, $\pi_D$ corresponds to \eqref{eq:d}, and $\pi_P$ corresponds to \eqref{eq:p}. Note that LP strong duality implies at optimality, $\hat{c}x^*=c^T \pi_D + b^T \pi_P$. where $x^*$ is a basic optimal solution to the \ref{p:PrLP} and from our assumption of uniqueness is subsequently a solution to the \ref{p:CLP}.
We now consider setting $\pi_D=x^*$, $\pi_O=1$, and $\pi_P = (A_B^T)^{-1}(\hat{c}-c)$ padded with zeros where $A_B$ corresponds to the columns of $A$ where $x^*$ is nonzero ($A_B$ must exist and form a basis from our uniqueness assumption). 
\begin{align}
    A^T\pi_P&=A_B^T\pi_P,\\
    &=\hat{c}-c, \\
    c^T\pi_D+b^T\pi_P &= c^Tx^*+(A_Bx^*_B)^T\pi_P, \\
    &= c^Tx^* + (x^*)^T(A_B^T\pi_P), \\
    &= c^Tx^*+(x^*)^T(\hat{c}-c), \\
    &= (x^*)^T\hat{c}.
\end{align}
Thus the objective values of \ref{p:CLP} and \ref{p:DCLP} are equivalent under such an assignment of variables and we have indeed found a description of an optimal dual solution. Applying Definition \ref{def:extlme} thus gives the following expression for a nodal LME:
\begin{align}
    \frac{d \ref{p:CLP}}{d P^D_i} &= \pi_O p_i + (\pi_P)_i, \\
    &= ((A_B^T)^{-1}(c))_i +((A_B^T)^{-1}(\hat{c}-c))_i, \\
    &= ((A_B^T)^{-1}(\hat{c}))_i,
\end{align}
where noting that $((A_B^T)^{-1}_i$ precisely describes the generation response with respect to a net injection change at node $i$ implies the equivalence.
 
\end{proof}
\begin{lemma}\label{lma:basiscontain}
    Any basis for the \ref{p:CLP} contains a basis of the \ref{p:PrLP}
\end{lemma}
\begin{proof}
    We begin by reformatting the \ref{p:CLP} to its standard form:
\begin{align}
    \min \quad &\hat{c}^T x \tag{Combined LP Standard} \label{p:CLPS}\\
    \text{s.t. }
    &Mv = w, \\
    &v \geq 0,
\end{align}
where 
\begin{equation*}
     M=\begin{bmatrix}
        A & 0 & 0 & 0 \\
        0 & A^T & -A^T & I \\
        -c^T &  b^T & -b^T  & 0
    \end{bmatrix},
    v = \begin{bmatrix}
        x \\ p^+ \\ p^- \\ s
    \end{bmatrix},
    w = \begin{bmatrix}
        b \\ c \\ 0
    \end{bmatrix}.
\end{equation*}
Note that any basis, $M_B\in \mathbb{R}^{(m+n+1)\times (n+3m)}$, for the \ref{p:CLPS} must contain a subset of columns of $M$ such that $M_B$ is full rank. Observing the first $m$ rows of $M$, we see that this requires selecting exactly $m$ columns from the first $n$ columns of $M$ such that the first $m$ rows of each column are linearly independent. Thus, the first $m\times m$ rows and columns of $M_B$ must be an independent subset of $A$. Note that this subset multiplied by the basis elements of $x$ must equal $b$ and thus precisely describes a basis for the \ref{p:PrLP}.
 
\end{proof}
\begin{lemma} \label{lma:basisclosed}
    Any basis for \ref{p:PrLP} which corresponds to an optimal solution for \ref{p:CLP} on the interval $b=(t,u)$ also corresponds to an optimal solution for \ref{p:CLP} on the interval $b=[t,u]$.
\end{lemma}

\begin{proof}
    Suppose not, suppose there exists some optimal basis $B$ for the \ref{p:PrLP} and some interval $(t,u)$ such that $B$ corresponds to an optimal solution for the \ref{p:CLP} on $(t,u)$ but not at $b=t$ or $b=u$. We begin by observing that $B$ being optimal on $(t,u)$ for \ref{p:CLP} implies that $B$ is optimal for \ref{p:PrLP} on $b=(t,u)$ as well. 

    Furthermore, we note that $B$ must also be an optimal basis for \ref{p:PrLP} at $b=t$ and $b=u$ as the conditions for the optimality of $B$ for \ref{p:PrLP} are described as 
    \begin{align}
        A_B^{-1}b &\geq 0 \tag{feasiblity} \label{eq:basis:feas}\\
        c_{B}^TA_B^{-1}A&\leq c^T \tag{dual feasibility} \label{eq:basis:opt}
    \end{align}
    where \eqref{eq:basis:opt} remains true for any value of $b$ and \eqref{eq:basis:feas} remains true on closed intervals of $b$. 
    
    Thus, it must be that the optimal value of the \ref{p:CLP} must be $\leq \hat{c}_B^T(A_B)^{-1}b$ on the interval $b=[t,u]$ as optimality for \ref{p:PrLP} implies $B$ corresponds to a feasible solution to \ref{p:CLP}. In particular, since we have assumed $B$ does not correspond to an optimal solution for \ref{p:CLP} at $b=t$ and $b=u$, this implies that the optimal value at $b=t$ and $b=u$ are strictly less than $\hat{c}_B^T(A_B)^{-1}t$, $\hat{c}^T_B(A_B)^{-1}u$ respectively.

    We now consider $\hat{x}_{b=v}$ to be the optimal solution to \ref{p:CLP} at $b=v$ and note that the above logic implies \begin{equation}
        \hat{c}^T\hat{x}_{b=t}<\hat{c}_B^T(A_B)^{-1}t,
    \end{equation}
    \begin{equation}
        c^T \hat{x}_{b=t}=c_B^T(A_B)^{-1}t.
    \end{equation}
    Define $w=\frac{t+u}{2}$ and consider the solution, $x$, as the convex combination of $\frac{1}{2}\hat{x}$ and $\frac{1}{2}$ the corresponding solution to $B$ at $b=u$. Note that 
    \begin{equation}
        Ax=\frac{1}{2}(A\hat{x}_{b=t}+A_B(A_B)^{-1}u)=\frac{1}{2}(t+u)=w.
    \end{equation}
    Furthermore, 
    \begin{equation}
        c^Tx = \frac{1}{2}(c^T\hat{x}_{b=t}+c^T_B(A_B)^{-1}u)=\frac{1}{2}(c^T_B)(A_B)^{-1}t+c^T_B(A_B)^{-1}u) = c^T_B(A_B)^{-1}w,
    \end{equation} 
    implying $x$ is optimal for \ref{p:PrLP} at $b=w$ as $B$ is an optimal basis for \ref{p:PrLP} on $(t,u)\ni w$. Thus, $x$ will be a feasible solution to \ref{p:CLP}. However, considering 
    \begin{equation}
        \hat{c}^Tx=\frac{1}{2}(\hat{c}^T\hat{x}_{b=t}+\hat{c}^T(A_B)^{-1}u)<\frac{1}{2}(\hat{c}^T(A_B)^{-1}t+\hat{c}^T(A_B)^{-1}u)=\hat{c}^T(A_B)^{-1}w
    \end{equation} yields a contradiction as $B$ was assumed to correspond to an optimal solution for \ref{p:CLP} on $b=(t,u)\ni w$. Repeating the same logic but for a convex combination of the solutions $x_{b=u}$ and $(A_B)^{-1}t$ similarly yields a contradiction for the non optimality of $B$ at $b=u$.

\end{proof}

\section{Statement on the use of Generative AI}
During the preparation of this work, the authors used Google Gemini to provide editing suggestions for word choice, flow, and clarity. Google Gemini was also used in order to generate code scaffolding for plots and figures. After using this tool, the authors reviewed and edited the content as needed and take full responsibility for the content of the publication.

\section*{Acknowledgments}
This work was funded by Meta. Additionally, L.C. was partially supported by the Dick and Jerry Smallwood Fellowship Fund during portions of this work. The authors also thank Thomas Lee for the generation and provision of ERCOT grid network and demand data.


\bibliographystyle{informs2014} 
\bibliography{refs, references}


\end{document}